\documentclass[11pt]{article}
\usepackage[margin=1in]{geometry}
\usepackage{mdframed}
\usepackage{amssymb}
\usepackage{amsmath}
\usepackage{graphicx}
\usepackage{bm}
\usepackage{subfigure}
\usepackage{grffile}
\usepackage{float}
\usepackage{breqn}
\usepackage{verbatim}
\usepackage{relsize}
\usepackage{color}

\newtheorem{theorem}{Theorem}[section]
\newtheorem{lemma}[theorem]{Lemma}

\newtheorem{corollary}[theorem]{Corollary}

\newenvironment{proof}[1][Proof]{\begin{trivlist}
\item[\hskip \labelsep {\bfseries #1}]}{\end{trivlist}}

\newcommand{\qed}{\nobreak \ifvmode \relax \else
      \ifdim\lastskip<1.5em \hskip-\lastskip
      \hskip1.5em plus0em minus0.5em \fi \nobreak
      \vrule height0.75em width0.5em depth0.25em\fi}

\title{Variational Multiscale Modeling with Discontinuous Subscales:\\ Analysis and Application to Scalar Transport}

\author{Christopher Coley and John A. Evans$^{*}$ \\
\textit{\small Ann and H.J. Smead Department of Aerospace Engineering Sciences, University of Colorado Boulder}\\
$^*$ \small Corresponding author.  \textit{E-mail address:} john.a.evans@colorado.edu}
\date{}

\begin{document}

\maketitle

\section*{Abstract}

\noindent We examine a variational multiscale method in which the unresolved fine-scales are approximated element-wise using a discontinuous Galerkin method.  We establish stability and convergence results for the methodology as applied to the scalar transport problem, and we prove that the method exhibits optimal convergence rates in the SUPG norm and is robust with respect to the P\'{e}clet number if the discontinuous subscale approximation space is sufficiently rich.  We apply the method to isogeometric NURBS discretizations of steady and unsteady transport problems, and the corresponding numerical results demonstrate that the method is stable and accurate in the advective limit even when low-order discontinuous subscale approximations are employed.\\

\section{Introduction}

Multiscale phenomena are ubiquitous in science and engineering applications.  For instance, turbulent fluid flows are characterized by a continuum of spatial and temporal scales, and even laminar fluid flows exhibit multiscale behavior in the form of boundary and shear layers.  Due to the widespread presence and impact of multiscale phenomena, there is a great demand for numerical methods which account for multiscale effects on the numerical solution.

The variational multiscale method was originally introduced as a theoretical framework for incorporating missing fine-scale effects into numerical problems governing coarse-scale behavior \cite{Bazilevs07b,Hughes95a,Hughes98,Hughes99,Hughes07}.  Construction of the variational multiscale method is simple: decompose the solution to a partial differential equation into a sum of coarse-scale and fine-scale components, determine the fine-scale component analytically in terms of the coarse-scale component, and solve for the coarse-scale component numerically.  The above scale decomposition is uniquely specified by identifying a projector from the space of all scales onto the coarse-scale subspace.  As a consequence, the coarse-scale component is guaranteed to best-fit the solution in a variational sense.

The primary challenge in the variational multiscale method is determining the fine-scale component in terms of the coarse-scale component.  Namely, the problem governing the behavior of the fine-scales is infinite-dimensional and nearly as difficult to solve as the original problem of interest.  Consequently,  in practice, the fine-scale problem must be approximated.  The simplest approximations are based on algebraic models which express the fine-scale component in terms of an algebraic expression of the coarse-scale residual.  In fact, classical stabilization approaches such as the streamline upwind Petrov Galerkin (SUPG) method \cite{Brooks82}, the Galerkin Least Squares (GLS) method \cite{Hughes86,Hughes89}, and the Douglas-Wang method \cite{Douglas89} may be viewed as algebraic models.  While these models have proven to be quite successful in capturing the mean effect of the unresolved fine-scales on coarse-scale behavior, they typically are unable to account for higher-order moments of fine-scale components \cite{Calo05}.  For example, it has been shown that algebraic models yield inaccurate representations of the subgrid stress tensor for turbulent incompressible flows \cite{Wang10}.

A more accurate approximation of the fine-scale problem may be obtained using differential models wherein a model differential equation is solved for the fine-scale component.  Perhaps the simplest differential model is the method of residual-free bubbles \cite{Brezzi92,Brezzi97,Brezzi94,Brezzi02,Sangalli03}.  In this approach, the fine-scale solution space is approximated by the space of bubbles over each element.  This yields an element-wise problem which can then be solved for the fine-scale solution field over the element.  It should be noted that these element-wise problems are still infinite-dimensional, so they must in turn be solved using a numerical method.  Moreover, the fine-scale solution over each element exhibits multiscale features including the presence of boundary layers over element boundaries, so any proposed method must be able to account for such features.  In prior work, researchers have turned to Galerkin's method to solve these problems with either (i) polynomial bubble functions and subgrid viscosity \cite{Brezzi00,Guermond99} or (ii) piecewise polynomial functions over Shiskin submeshes \cite{Brezzi03,Brezzi98}.  The first approach leads to a simple implementation, though the exact value of the required subgrid viscosity is problem-dependent and polynomial bubble functions are unable to accurately represent layers.  The second approach leads to a more complex implementation, but it is better able to accurately capture boundary layer phenomena.

In this paper, we examine an alternative approach, which we refer to henceforth as the method of discontinuous subscales, in which the fine-scales in the method of residual-free bubbles are approximated element-wise using a discontinuous Galerkin method \cite{Arnold02}.  As the discontinuous subscales are not required to meet the residual-free bubble boundary conditions in a strong sense, they are able to represent layers without resorting to a complicated submesh.  Herein, we employ the symmetric interior penalty method to approximate diffusive fluxes and the upwind method to approximate advective fluxes on element boundaries in the element-wise fine-scale problems \cite{Arnold82,Shu2009}.  This yields a stable methodology for transport and incompressible fluid flow even in the advective limit, though it should be mentioned that alternative discontinuous Galerkin methods may be employed \cite{Cockburn98,Cockburn01,Nguyen09}.  To examine the effectiveness of the method of discontinuous subscales, we conduct a theoretical stability and convergence analysis of the method as applied to the scalar transport problem.  Though the scalar transport problem is linear, it exhibits multiscale behavior in the form of boundary and internal layers and thus serves as a simplified vehicle to study the more complicated Navier-Stokes equations.  Through our analysis, we find that the method of discontinuous subscales exhibits optimal convergence rates and is robust with respect to the P\'{e}clet number if the discontinuous subscale approximation space is sufficiently rich.  We further explicitly identify the required richness in the context of affine finite elements.  We finish this paper by applying the method to isogeometric discretizations \cite{Hughes05} of steady and unsteady transport problems.  Our motivation for examining isogeometric discretizations is that they exhibit improved stability and accuracy properties as compared with classical finite elements when applied to difficult transport \cite{Bazilevs06}, laminar and turbulent fluid flow \cite{Bazilevs07b,Evans13a,Evans13b,vanOpstal17}, and fluid-structure interaction \cite{Bazilevs08} applications.  Surprisingly, we find that the method of discontinuous subscales is stable and accurate in the advective limit even when lowest-order discontinuous subscale approximations are employed as a companion fine-scale model for an isogeometric discretization.  

It should be mentioned that the method of discontinuous subscales is related to many other multiscale methods in the literature, and in particular, it is closely aligned with the discontinuous residual-free bubble method of Sangalli \cite{Sangalli04}, the multiscale discontinuous Galerkin method of Hughes \textit{et al.} \cite{Bochev06,Buffa06,Hughes06}, and the parameter-free variational multiscale method of Cottrell \cite{Cottrell07}.  In the discontinuous residual-free bubble method, one also employs a discontinuous Galerkin method to approximate fine-scale solution behavior, though alternative boundary conditions are employed at the element boundaries.  Moreover, no theoretical stability or convergence results exist for the method.  In the multiscale discontinuous Galerkin method, variational multiscale analysis and an interscale transfer operator are employed to essentially reduce the computational complexity of a discontinuous Galerkin method to that of a continuous Galerkin method.  While the multiscale discontinuous Galerkin method has a firm theoretical foundation, it is fairly difficult to implement, and it is not possible to employ different polynomial degrees for the coarse- and fine-scale components in the method.  The parameter-free variational multiscale method is perhaps most closely aligned with the method of discontinuous subscales.  In this approach, the discontinuous Galerkin method is employed to approximate the so-called fine-scale Green's function as opposed to the residual-free bubble directly.  No theoretical stability or convergence results exist for the parameter-free variational multiscale method, though we believe that the analysis carried out here can be easily extended to this approach.

An outline of the remainder of the paper is as follows.  In Section 2, we provide an overview of the scalar transport problem.  In Section 3, we introduce the variational multiscale approach, the method of residual-free bubbles, and the method of discontinuous subscales.  In Section 4, we present a theoretical analysis of the method of discontinuous subscales, and in Section 5, we apply the method to a selection of steady and unsteady scalar transport problems.  Finally, in Section 6, we provide concluding remarks.

\section{The Governing Equations of Scalar Transport}
We limit our discussion in this paper to the scalar transport problem, also referred to as the advection-diffusion or drift-diffusion problem in the literature.  Though we consider both steady and unsteady transport throughout, we only state the strong and weak forms of the unsteady problem here as the corresponding forms for the steady problem are implied.  With the above in mind, the strong form of the unsteady scalar transport problem reads as follows: Find $u: \overline{\Omega} \times \left[0,T\right) \rightarrow \mathbb{R}$ such that:
\begin{equation}\label{eq:ad_strong}
\begin{aligned}
\frac{\partial u}{\partial t} + \bm{a} \cdot \nabla u - \nabla \cdot \left( \kappa \nabla u \right)  & = f & \text{in}\ \Omega \times \left(0,T\right) \\
u & = g & \text{on}\ \Gamma_D \times \left(0,T\right) \\
\kappa \nabla u \cdot \bm{n} & = h & \text{on}\ \Gamma_N \times \left(0,T\right) \\
u & = u_0 & \text{in}\ \Omega \times \{0\}
\end{aligned}
\end{equation}
where $\Omega \subset \mathbb{R}^d$ is the spatial domain of the problem, $d$ is the number of spatial dimensions, $\Gamma$ is the boundary of the domain, $\Gamma_D$ is the Dirichlet part of the boundary, $\Gamma_N$ is the Neumann part of the boundary, $\bm{n}$ is the unit outward boundary normal, $u$ is the scalar being transported,  $\bm{a}: \Omega \rightarrow \mathbb{R}^d$ is the advective velocity, $\kappa: \Omega \rightarrow \mathbb{R}_+$ is the diffusivity, $f: \Omega \rightarrow \mathbb{R}$ is the applied forcing, $g: \Gamma_D \rightarrow \mathbb{R}$ is the prescribed Dirichlet data, $h: \Gamma_N \rightarrow \mathbb{R}$ is the prescribed Neumann data, and $u_0 : \Omega \rightarrow \mathbb{R}$ is the prescribed initial data.  In order for the scalar transport problem to be well-posed, we require that $\Gamma = \overline{\Gamma_D \cap \Gamma_N}$ and that $a_n = \bm{a} \cdot \bm{n} \geq 0$ on $\Gamma_N$. Note that to simplify our presentation, we have assumed that the advective velocity, diffusivity, forcing, and prescribed data are all independent of time.  We further assume throughout that these quantities are all smooth and that the advective velocity is divergence-free.

To establish a weak form of the unsteady scalar transport problem, we must first define a test space and set of trial functions.  With this in mind, let:
\begin{equation}
\mathcal{V} := \left\{ v \in H^1\left(\Omega\right): v|_{\Gamma_D} = 0 \right\}
\end{equation}
and:
\begin{equation}
\mathcal{S} := \left\{ u \in H^1\left(\Omega\right): u|_{\Gamma_D} = g \right\}
\end{equation}
denote the time-instaneous test space and set of trial functions, respectively, where $H^1\left(\Omega\right)$ is the Sobolev space of square-integrable functions with square-integrable derivatives.  We further define the space-time set of trial functions as:
\begin{equation}
\mathcal{S}_T := C([0,T];\mathcal{S})
\end{equation}
which is the set of continuous functions $u: [0,t] \rightarrow \mathcal{S}$ with:
$$\max_{0 \leq t \leq T} \| u(t) \|_{H^1(\Omega)} < \infty.$$
The weak form of the unsteady advection-diffusion problem then reads as follows: Find $u \in \mathcal{S}_T$ such that $u(\cdot,0) = u_0(\cdot)$ and:
\begin{equation}\label{eq:ad_weak}
\begin{aligned}
\int_\Omega \frac{\partial u}{\partial t} v -  \int_\Omega u \bm{a} \cdot \nabla v +  \int_\Omega \kappa \nabla u \cdot \nabla v + \int_{\Gamma_N} a_n u v = \int_\Omega f v + \int_{\Gamma_N} h v
\end{aligned}
\end{equation}
\noindent for all $v \in \mathcal{V}$ and almost every $t \in (0,T)$.  It is well-known that the strong and weak forms of the advection-diffusion problem admit a unique solution that depends smoothly on the prescribed data \cite{Evans10}.

To simplify our later discussion, we define the following operator:
\begin{equation}\label{eq:l_operators}
\begin{aligned}
\mathcal{L}_t & := \frac{\partial}{\partial t} + \bm{a} \cdot \nabla - \nabla \cdot \kappa \nabla
\end{aligned}
\end{equation}
Note that $\mathcal{L}_t$ is precisely the differential operator associated with the unsteady scalar transport problem.

\section{Variational Multiscale Analysis, Residual-Free Bubbles, and the Method of Discontinuous Subscales}
Now that we have defined the governing equations for scalar transport, we are ready to present our numerical method.  Before doing so, however, we first review the variational multiscale method followed by the residual-free bubble method.

\subsection{The Variational Multiscale Method}
In the variational multiscale method, the solution to a partial differential equation is split into a finite-dimensional, coarse-scale component and an infinite-dimensional, fine-scale component through the use of variational projection \cite{Hughes07}.  As such, the coarse-scale component is \textit{a priori} guaranteed to best-fit the exact solution in a variational sense.

To proceed forward, we must first define a finite-dimensional, coarse-scale test space $\bar{\mathcal{V}} \subset \mathcal{V}$ and a corresponding continuous, linear projection operator $\mathcal{P}: \mathcal{V} \rightarrow \bar{\mathcal{V}}$.  The projection operator naturally splits the test space into coarse-scale and fine-scale components as exhibited by the decomposition:
$$\mathcal{V} = \bar{\mathcal{V}} \oplus \mathcal{V}'$$
where $\mathcal{V}' = \ker(\mathcal{P})$ is the infinite-dimensional, fine-scale test space.  Consequently, each test function $v \in \mathcal{V}$ is uniquely represented as the sum of a coarse-scale test function $\bar{v} = \mathcal{P}v \in \bar{\mathcal{V}}$ and a fine-scale test function $v' = v - \bar{v} \in \mathcal{V}$.  Associated with the coarse-scale test space is a corresponding set of trial functions of the form $\bar{\mathcal{S}} = \bar{g} + \bar{\mathcal{V}}$ where $\bar{g} \in \mathcal{S}$ satisfies the non-homogeneous Dirichlet boundary condition $\bar{g}|_{\Gamma_D} = g$.  This inspires a similar split of the set of trial functions into coarse-scale and fine-scale components:
$$\mathcal{S} = \bar{\mathcal{S}} + \mathcal{S}'$$
where $\mathcal{S}'$ is the infinite-dimensional set of fine-scale trial functions.  Therefore, the solution to the scalar transport problem $u \in \mathcal{S}$ is uniquely represented as the sum of a coarse-scale component $\bar{u} = \bar{g} + \mathcal{P}u \in \bar{\mathcal{S}}$ and a fine-scale component $u' = u - \bar{u} \in \mathcal{S}'$.  Since the coarse-scale trial functions satisfy the required non-homogeneous boundary condition, the fine-scale trial functions satisfy homogeneous boundary conditions, and thus we have $\mathcal{S}' = \mathcal{V}'$.  

Heretofore, we have discussed how to split the solution to the scalar transport problem into coarse-scale and fine-scale components, but we have not discussed how one may obtain said components via a numerical method.  To do so, we simply use the decomposition $\mathcal{V} = \bar{\mathcal{V}} \oplus \mathcal{V}'$ and bilinearity to perform a scale splitting of the scalar transport problem.  The corresponding variational problem takes the form: Find $\bar{u} \in C\left([0,T];\bar{\mathcal{S}}\right)$ and $u' \in C\left([0,T];\mathcal{S}'\right)$ such that:
\begin{equation}\label{eq:ad_coarse}
\begin{aligned}
\int_\Omega \frac{\partial \left(\bar{u} + u'\right)}{\partial t} \bar{v} -  \int_\Omega \left(\bar{u} + u'\right) \bm{a} \cdot \nabla \bar{v} +  \int_\Omega \kappa \nabla \left(\bar{u} + u'\right) \cdot \nabla \bar{v} + \int_{\Gamma_N} a_n \left(\bar{u} + u'\right) \bar{v} = \int_\Omega f \bar{v} + \int_{\Gamma_N} h \bar{v}
\end{aligned}
\end{equation}
for all $\bar{v} \in \bar{\mathcal{V}}$ and almost every $t \in (0,T)$ and:
\begin{equation}\label{eq:ad_fine}
\begin{aligned}
\int_\Omega \frac{\partial \left(\bar{u} + u'\right)}{\partial t} v' -  \int_\Omega \left(\bar{u} + u'\right) \bm{a} \cdot \nabla v' +  \int_\Omega \kappa \nabla \left(\bar{u} + u'\right) \cdot \nabla v' + \int_{\Gamma_N} a_n \left(\bar{u} + u'\right) v' = \int_\Omega f v' + \int_{\Gamma_N} h v'
\end{aligned}
\end{equation}
for all $v' \in \mathcal{V}'$ and almost every $t \in (0,T)$.  Eq. \eqref{eq:ad_coarse} is referred to as the coarse-scale problem and Eq. \eqref{eq:ad_fine} is referred to as the fine-scale problem.  One can solve the fine-scale problem for $u' \in \mathcal{S}'$ in terms of the so-called coarse-scale residual:
\begin{equation}
\textup{Res}\left(\bar{u}\right) := f - \mathcal{L}_t \bar{u}
\end{equation}
and insert the resulting solution back into the coarse-scale problem in order to arrive at a final finite-dimensional system for the coarse-scale solution $\bar{u} \in \bar{\mathcal{S}}$.

The primary issue associated with the variational multiscale method is that $\mathcal{V}'$ is an infinite-dimensional space and thus solving the fine-scale problem is an intractable task.  Fortunately, for most problems of interest, it is sufficient to approximate the effect of the fine-scales on the coarse-scale solution in order to produce stable and accurate numerical solutions \cite{Cottrell07}.  With this in mind, we now turn our attention to differential approaches for modeling the fine-scale problem with specific reference to the method of residual-free bubbles.

\subsection{Differential Fine-Scale Modeling with Residual-Free Bubbles}

Assume that we have a decomposition of the spatial domain into a mesh of elements $\mathcal{M} = \left\{ \Omega_K \right\}_{K=1}^{n_{el}}$ satisfying $\overline{\cup_K \Omega_K} = \overline{\Omega}$.  We associate the decomposition with the finite element mesh in the finite element setting, and we associate the decomposition with the B\'{e}zier mesh in the context of isogeometric analysis \cite{Borden11}.  In the residual-free bubble approach \cite{Brezzi92,Brezzi97,Brezzi94}, we simply replace the fine-scale test space and set of trial functions in the variational multiscale method by the space of $H^1$-conforming bubbles over each element:
\begin{equation}
\mathcal{V}'_{bubble} = \mathcal{S}'_{bubble} := \left\{ v \in \mathcal{V}: v \in H^1_0(\Omega_K) \textup{ for each element } \Omega_K \subset \Omega \right\}.
\end{equation}
Since the functions in $\mathcal{V}'_{bubble}$ and $\mathcal{S}'_{bubble}$ are zero-valued on the boundary of each element, the fine-scale problem applies element-by-element.  Thus, we have:
\begin{equation}\label{eq:ad_split_RFB_element}
\int_{\Omega_K} \frac{\partial \left(\bar{u} + u'\right)}{\partial t} v' -  \int_{\Omega_K} \left(\bar{u} + u'\right) \bm{a} \cdot \nabla v' +  \int_{\Omega_K} \kappa \nabla \left(\bar{u} + u'\right) \cdot \nabla v' = \int_{\Omega_K} f v'
\end{equation}
for every $v' \in H^1_0(\Omega_K)$ and each element $\Omega_K \subset \Omega$. By integration-by-parts, the above implies the following element-wise problem:

\begin{dmath} \label{eq:element_fine_RFB}
\begin{aligned}
\mathcal{L}_t u'_K & = \textup{Res}\left(\bar{u}\right)\ \text{in}\ \Omega_K \\
u'_K & = 0\ \text{on}\ \Gamma_K,
\end{aligned}
\end{dmath}

\noindent where $u'_K = u'|_{\Omega_K}$ and  $\Gamma_K$ is the boundary of the element $\Omega_K$.  Residual-free bubbles are defined as the functions $u'_K$ which satisfy the above problem strongly.

Note that the governing equations for the residual-free bubbles are still infinite-dimensional.  While these equations are easier to solve than the fine-scale problem associated with the variational multiscale method, they are still nearly as difficult to solve as the original scalar transport problem.  Thus, in practice, one must turn to a numerical method to approximate the residual-free bubbles.  Additionally, the residual-free bubbles exhibit multiscale behavior such as boundary layers, so any method must be able to account for such behavior.  As mentioned in the introduction, a number of approaches have been proposed in prior work based on the continuous Galerkin method.  While these approaches have been shown to be closely related to the SUPG method, they exhibit complex features such as problem-dependent subgrid diffusivity \cite{Brezzi00} or advection-aligned Shishkin subgrid meshes \cite{Brezzi03} in order to overcome the advective instabilities associated with the continuous Galerkin method and account for the multiscale features in the residual-free bubbles.  We instead turn to another means of approximating the residual-free bubbles in the next subsection.

\subsection{Approximation of the Residual-Free Bubbles with Discontinuous Subscales}

Recall that the residual-free bubbles exhibit the same sorts of multiscale phenomena as the solution to the scalar transport problem,  including boundary layers along element boundaries.  This inspires us to weaken the enforcement of element-wise homogeneous Dirichlet boundary conditions via a discontinuous Galerkin methodology.  We refer to this approach as the method of discontinuous subscales as the corresponding residual-free bubble approximations are discontinuous between adjacent elements.

To set the stage, we need to first define a finite-dimensional test space and set of trial functions for the discontinuous subscales.  As with the method of residual-free bubbles, these two sets collapse to the same space, and we simply require that the test and trial functions are $H^1$-conforming over each element.  We denote the discontinuous subscale approximation space over element $\Omega_K$ as $\mathcal{V}'_{disc,K} = \mathcal{S}^'_{disc,K} \subset H^1\left(\Omega_K\right)$ and the approximation space over the entire domain as:
\begin{equation}
\mathcal{V}'_{disc} = \mathcal{S}'_{disc} := \left\{ v \in L^2(\Omega): v|_{\Omega_K} \in \mathcal{V}'_{disc,K} \textup{ for every } \Omega_K \in \mathcal{M} \right\}.
\end{equation}
There are many potential candidates for the discontinuous subscale approximation space, and we consider polynomial basis functions of a particular degree later in this paper.

It remains now to discretize the governing residual-free bubble equation using our discontinuous subscale approximation space.  We turn to a discontinuous Galerkin method in which the symmetric interior penalty method is employed to approximate the diffusive fluxes on element boundaries and the upwind method is employed to approximate the advective fluxes on element boundaries \cite{Shu2009}.  This approach is known to be stable even in the advective limit, and it results in the following element-wise problem:
\begin{equation}
\begin{aligned}
\int_{\Omega_K} \frac{\partial u'}{\partial t} v' -  \int_{\Omega_K} u' \bm{a} \cdot \nabla v' + \int_{\Omega_K} \kappa \nabla u' \cdot \nabla v' \\ + \int_{\Gamma_K} u' v' a_n^+ - \int_{\Gamma_K} \kappa \nabla u' \cdot \bm{n} v' - \int_{\Gamma_K} u' \kappa \nabla v' \cdot \bm{n} + \int_{\Gamma_K} \frac{C_{pen} \kappa}{h_K} u' v' &= \int_{\Omega_K} \textup{Res}\left(\bar{u}\right) v'
\end{aligned}
\end{equation}
for every $v' \in \mathcal{V}'_{disc,K}$ where $a^+_n = \left(\bm{a} \cdot \bm{n}\right)^+ = \max\left\{ \bm{a} \cdot \bm{n}, 0 \right\}$.  Above, $h_K$ indicates a measure of the size of element $\Omega_K$ (e.g., $h_K$ could be the diameter of the smallest ball encompassing the element), and $C_{pen} > 0$ is a stabilization parameter that must be specifically chosen such that the resulting methodology is stable.  It should be noted that the exact value of the stabilization parameter should depend on both the coarse-scale and discontinuous subscale approximation spaces, and we discuss how to choose the stabilization parameter in the next section.

At this juncture, we have arrived at a fully specified method.  However, we should highlight a few modifications that yield a more robust and streamlined methodology.  First of all, we have found that the optional inclusion of an artificial diffusivity $\kappa_{art}$ in the subscale governing equation generally improves stability as well as the conditioning of the resulting linear system.  Guidelines for how to choose the artificial diffusivity are provided in the next section.  Second of all, we have found that neglecting the influence of the subscale solution field in the unsteady and diffusive terms appearing in the coarse-scale governing equation also improves stability.  It should be mentioned that these terms also vanish if the projection operator is specially chosen as observed in previous works \cite{Hughes07}.  Finally, one may choose to ignore the time-history of the subscale solution field by removing the time-derivative of the subscales appearing in the subscale governing equation.  This approach, which we refer to as the quasi-static model, leads to a simpler implementation than the corrresponding dynamic model.  However, both the dynamic and quasi-static models exhibit nearly the same computational cost.  It should be noted that the dynamic and quasi-static labels are inspired by the Dynamic Subscales (DSS) and Quasi-Static Subscales (QSS) methods of Codina \textit{et al.} \cite{Codina07}, though the DSS and QSS methods do not attempt to solve for the residual-free bubbles.

Collecting all of our governing equations together, we obtain the following:\\

\noindent \textbf{Coarse-Scale Governing Equation:}
\begin{equation}
\begin{aligned}
\int_\Omega \frac{\partial \bar{u}}{\partial t} \bar{v} -  \int_\Omega \left(\bar{u} + u'\right) \bm{a} \cdot \nabla \bar{v} +  \int_\Omega \kappa \nabla \bar{u} \cdot \nabla \bar{v} + \int_{\Gamma_N} a_n \bar{u} \bar{v} = \int_\Omega f \bar{v} + \int_{\Gamma_N} h \bar{v} \\ \textup{ for all } \bar{v} \in \bar{\mathcal{V}}
\end{aligned}
\end{equation}

\noindent \textbf{Dynamic Discontinuous Subscale Model:}
\begin{equation}
\begin{aligned}
\int_{\Omega_K} \frac{\partial u'}{\partial t} v' -  \int_{\Omega_K} u' \bm{a} \cdot \nabla v' + \int_{\Omega_K} \left( \kappa + \kappa_{art} \right) \nabla u' \cdot \nabla v' \\ + \int_{\Gamma_K} u' v' a_n^+ - \int_{\Gamma_K} \kappa \nabla u' \cdot \bm{n} v' - \int_{\Gamma_K} u' \kappa \nabla v' \cdot \bm{n} + \int_{\Gamma_K} \frac{C_{pen} \kappa}{h_K} u' v' &= \int_{\Omega_K} \textup{Res}\left(\bar{u}\right) v' \\ & \textup{ for all } v' \in \mathcal{V}'_{disc,K} \textup{ and } \Omega_K \in \mathcal{M}
\end{aligned}
\end{equation}

\noindent \textbf{Quasi-Static Discontinuous Subscale Model:}
\begin{equation}
\begin{aligned}
\int_{\Omega_K} u' \bm{a} \cdot \nabla v' +  \int_{\Omega_K} \left(\kappa + \kappa_{art} \right) \nabla u' \cdot \nabla v' \cdot \bm{n}\\ + \int_{\Gamma_K} u' v' a_n^+ - \int_{\Gamma_K} \kappa \nabla u' \cdot \bm{n} v' - \int_{\Gamma_K} u' \kappa \nabla v' \cdot \bm{n} + \int_{\Gamma_K} \frac{C_{pen} \kappa}{h_K} u' v' &= \int_{\Omega_K} \textup{Res}\left(\bar{u}\right) v' \\ & \textup{ for all } v' \in \mathcal{V}'_{disc,K} \textup{ and } \Omega_K \in \mathcal{M}
\end{aligned}
\end{equation}

The governing semi-discrete equations may be discretized in time using any particular time-integrator of interest, yielding a linear system to be solved at each time-step.  Note that in both the dynamic and quasi-static discontinuous subscale models, the subscale solution fields on each element are decoupled.  Consequently, an element-wise static condensation procedure can be employed to remove the subscale degrees-of-freedom from the discrete system at each time-step, yielding a reduced linear system for the coarse-scale solution field \cite{Wilson74}.  Moreover, as the subscale solution fields on each element are decoupled, this reduced linear system has exactly the same sparsity pattern as that associated with either Galerkin's method or the SUPG method.

\section{Analysis of the Method of Discontinuous Subscales} \label{analysis}

\noindent Now that we have presented our methodology for solving the scalar transport problem, we establish stability and convergence results.  To simplify exposition, we deal solely with the steady advection-diffusion problem with homogeneous Dirichlet boundary conditions.  With this in mind, let us define the group vector space $\bm{\mathcal{V}}^h = \overline{\mathcal{V}} \times \mathcal{V}'_{disc}$, the group variables $\textbf{U}^h = \left( \bar{u}, u' \right) \in \bm{\mathcal{V}}^h$ and $\textbf{V}^h = \left( \bar{v}, v'\right) \in \bm{\mathcal{V}}^h$, the group bilinear form:
\begin{equation}
\mathbf{B}\left(\textbf{U}^h,\textbf{V}^h\right) := \bar{B}(\bar{u},\bar{v}) + \bar{C}(u',\bar{v}) + C'(\bar{u},v') + B'(u',v')
\end{equation}
where:
\begin{eqnarray}
\bar{B}(\bar{u},\bar{v}) & := & -\int_{\Omega} \bar{u} \bm{a} \cdot \nabla \bar{v} + \int_{\Omega} \kappa \nabla \bar{u} \cdot \nabla \bar{v} \nonumber \\
\bar{C}(u',\bar{v}) & := & -\sum_{K=1}^{n_{el}} \int_{\Omega_K} u' \bm{a} \cdot \nabla \bar{v} \nonumber \\
C'(\bar{u},v') & := & \sum_{K=1}^{n_{el}} \int_{\Omega_K} v' \left( \bm{a} \cdot \nabla \bar{u} - \nabla \cdot \left( \kappa \nabla \bar{u} \right) \right) \nonumber \\
B'(u',v') & := & \sum_{K=1}^{n_{el}} \left( -\int_{\Omega_K} u' \bm{a} \cdot \nabla v' + \int_{\Omega_K} \left( \kappa + \kappa_{art} \right) \nabla u' \cdot \nabla v' + \int_{\Gamma_K} u' v' a_n^+ \right. \nonumber \\
& & \left. - \int_{\Gamma_K} \kappa \nabla u' \cdot \bm{n} v' - \int_{\Gamma_K} u' \kappa \nabla v' \cdot \bm{n} + \int_{\Gamma_K} \frac{C_{pen} \kappa}{h_K} u' v' \right), \nonumber
\end{eqnarray}
and the group linear form:
\begin{equation}
\mathbf{L}\left(\textbf{W}^h\right) := \bar{L}(\bar{v}) + L'(v')
\end{equation}
where:
\begin{eqnarray}
\bar{L}(\bar{v}) & := & \int_{\Omega} f \bar{v} \nonumber \\
L'(v') & := & \int_{\Omega} f v' \nonumber.
\end{eqnarray}
With the above notation established, our discrete problem is as follows: Find $\textbf{U}^h \in \bm{\mathcal{V}}^h$ such that for all $\textbf{V}^h \in \bm{\mathcal{V}}^h$:
\begin{equation}
\mathbf{B}\left(\textbf{U}^h,\textbf{V}^h\right) = \mathbf{L}\left(\textbf{V}^h\right).
\end{equation}
Throughout, we assume that $\overline{\mathcal{V}}$ is a finite-dimensional space of continuous piecewise polynomial or tensor-product polynomial functions of degree $p$ defined over a given mesh $\mathcal{M}$ of simplices (triangles and tetrahedra) or parallelotopes (quadrilaterals and hexahedra).  We also assume that $\mathcal{V}'_{disc}$ is a finite dimensional-space of discontinuous piecewise polynomial or tensor-product polynomials of degree $p_f$ defined over the same mesh $\mathcal{M}$.  While our analysis only strictly covers the setting of affinely-mapped finite elements, it easily extends to the more general settings of curvilinear finite elements and isogeometric analysis.

Throughout this section, we make use of the classical Lebesgue spaces $L^q(D)$ endowed with the norm $\| \cdot \|_{L^q(D)}$ where $1 \leq q \leq \infty$ and $D \subset \mathbb{R}^d$ is a generic open domain for integer $d \geq 1$.  We will also utilize the Sobolev spaces $W^{k,q}(D)$ for $k$ a non-negative integer and $1 \leq q \leq \infty$, endowed with the norm:
\begin{equation}
\begin{aligned}
\| u \|_{k,q,D} &:=& \left( \sum_{\alpha_1 + \ldots + \alpha_d \leq k} \left\| \frac{\partial^{\alpha_1}}{\partial x^{\alpha_1}_1} \ldots \frac{\partial^{\alpha_d}}{\partial x^{\alpha_d}_d} u \right\|^2_{L^q(D)} \right)^{1/2}
\end{aligned}
\end{equation}
and semi-norm:
\begin{equation}
\begin{aligned}
| u |_{k,q,D} &:=& \left( \sum_{\alpha_1 + \ldots + \alpha_d = k} \left\| \frac{\partial^{\alpha_1}}{\partial x^{\alpha_1}_1} \ldots \frac{\partial^{\alpha_d}}{\partial x^{\alpha_d}_d} u \right\|^2_{L^q(D)} \right)^{1/2}.
\end{aligned}
\end{equation}
In the setting when $q = 2$, the Sobolev spaces $W^{k,q}(D)$ become $H^k(D)$ and we use the simplified notation $\| \cdot \|_{k,D} \equiv \| \cdot \|_{k,2,D}$ and $| \cdot |_{k,D} \equiv | \cdot |_{k,2,D}$.  Sobolev norms are defined over boundaries of open domains in an analogous manner.

To proceed forward, we must make certain assumptions regarding the form of the penalty constant $C_{pen}$ and the artificial diffusivity $\kappa_{art}$.  In particular, we assume the following:\\

\noindent \textbf{Assumption 1:} \textit{The penalty constant $C_{pen}$ satisfies $C_{pen} \geq 8 C_{trace}$ where $C_{trace} > 0$ is a sufficiently large positive constant such that:}
\begin{eqnarray}
\| \bar{v} \|^2_{0,\Gamma_K} \leq \frac{C_{trace}}{h_K} \| \bar{v} \|^2_{0,\Omega_K} \hspace{15pt} \textup{and} \hspace{15pt} \| \nabla \bar{v} \cdot \bm{n} \|^2_{0,\Gamma_K} \leq \frac{C_{trace}}{h_K} | \bar{v} |^2_{1,\Omega_K} \nonumber \\
\| v' \|^2_{0,\Gamma_K} \leq \frac{C_{trace}}{h_K} \| v' \|^2_{0,\Omega_K} \hspace{9pt} \textup{and} \hspace{9pt} \| \nabla v' \cdot \bm{n} \|^2_{0,\Gamma_K} \leq \frac{C_{trace}}{h_K} | v' |^2_{1,\Omega_K} \nonumber
\end{eqnarray}
\textit{for every $\bar{v} \in \overline{\mathcal{V}}$, $v' \in \mathcal{V}'_{disc}$, and element $\Omega_K \in \mathcal{M}$.}\\

\noindent \textbf{Assumption 2:} \textit{The artificial diffusivity $\kappa_{art}$ takes the form:}
\begin{equation}
\kappa_{art} = C_{art} \left( \textup{max}\left\{ \frac{\| \bm{a} \|_{0,\infty,\Omega_K}}{h_K}, \frac{C_{inv} \kappa}{h^2_K} \right \}\right)^{-1} \| \bm{a} \|^2_{\infty,\Omega_K} \nonumber
\end{equation}
\textit{where $C_{art} \geq 0$ is an arbitrary non-negative constant and $C_{inv} > 0$ is a sufficiently large positive constant such that:}
\begin{eqnarray}
| \bar{v} |^2_{1,\Omega_K} \leq \frac{C_{inv}}{h^2_K} \| \bar{v} \|^2_{0,\Omega_K} \hspace{15pt} \textup{and} \hspace{15pt} \| \Delta \bar{v} \|^2_{0,\Omega_K} \leq \frac{C_{inv}}{h^2_K} | \bar{v} |^2_{1,\Omega_K} \nonumber \\
| v' |^2_{1,\Omega_K} \leq \frac{C_{inv}}{h^2_K} \| v' \|^2_{0,\Omega_K}  \hspace{9pt} \textup{and} \hspace{9pt} \| \Delta v' \|^2_{0,\Omega_K} \leq \frac{C_{inv}}{h^2_K} | v' |^2_{1,\Omega_K} \nonumber
\end{eqnarray}
\textit{for every $\bar{v} \in \overline{\mathcal{V}}$, $v' \in \mathcal{V}'_{disc}$, and element $\Omega_K \in \mathcal{M}$.}\\

\noindent One can readily find the trace and inverse constants associated with Assumptions 1 and 2 by solving element-wise eigenproblems.  Alternatively, explicit bounds for these constants are available both in the setting of finite elements and isogeometric analysis \cite{Bazilevs06,Evans13,Harari92,Warburton03}.

The following result shows that the group bilinear form is coercive, and hence the group variable solution is unique.

\begin{theorem} \label{theorem1}
The coercivity result
\begin{equation}
\textup{\textbf{B}}(\textup{\textbf{V}}^h,\textup{\textbf{V}}^h) \geq \frac{1}{4} \left( \kappa | \bar{v} |^2_{1,\Omega} + \sum_{K=1}^{n_{el}} \left( \left( \kappa + \kappa_{art} \right) | v' |^2_{1,\Omega_K} + \frac{C_{pen}\kappa}{h_K} \| v' \|^2_{0,\Gamma_K} + \| a_n^{1/2} v' \|^2_{0,\Gamma_K} \right) \right) \nonumber
\end{equation}
holds for all $\textup{\textbf{V}}^h \in \bm{\mathcal{V}}^h$ provided Assumption 1 is satisfied.
\end{theorem}

\begin{proof}
For the coarse-scale bilinear form, we evaluate:
\begin{eqnarray}
\bar{B}(\bar{v},\bar{v}) & = & -\int_{\Omega} \bar{v} \bm{a} \cdot \nabla \bar{v} + \int_{\Omega} \kappa \nabla \bar{v} \cdot \nabla \bar{v} \nonumber \\
& = & - \int_{\Omega} \frac{1}{2} \nabla \cdot \left( \bm{a} \bar{v}^2 \right) + \kappa | \bar{v} |^2_{1,\Omega} \hspace{15pt} \textup{(by the product rule)} \nonumber \\
& = & \kappa | \bar{v} |^2_{1,\Omega}. \hspace{102pt} \textup{(by the divergence theorem)} \label{eq:th1_begin}
\end{eqnarray}
Similarly, for the fine-scale bilinear form, we evaluate:
\begin{eqnarray}
B'(v',v') & = & \sum_{K=1}^{n_{el}} \left( -\int_{\Omega_K} v' \bm{a} \cdot \nabla v' + \int_{\Omega_K} \left( \kappa + \kappa_{art} \right) \nabla v' \cdot \nabla v' + \int_{\Gamma_K} v' v' a_n^+ \right. \nonumber \\
& & \left. - \int_{\Gamma_K} \kappa \nabla v' \cdot \bm{n} v' - \int_{\Gamma_K} v' \kappa \nabla v' \cdot \bm{n} + \int_{\Gamma_K} \frac{C_{pen} \kappa}{h_K} v' v' \right) \nonumber \\
 & = & \sum_{K=1}^{n_{el}} \left( \left( \kappa + \kappa_{art} \right) | v' |^2_{1,\Omega_K} + \frac{C_{pen}\kappa}{h_K} \| v' \|^2_{0,\Gamma_K} \right. \nonumber \\
& & \left. -\int_{\Omega_K} \frac{1}{2} \nabla \cdot \left( \bm{a} \left(v'\right)^2 \right) + \int_{\Gamma_K} a_n^+ \left(v'\right)^2 - 2 \int_{\Gamma_K} \kappa \nabla v' \cdot \bm{n} v' \right) \nonumber \\
 & = & \sum_{K=1}^{n_{el}} \left( \left( \kappa + \kappa_{art} \right) | v' |^2_{1,\Omega_K} + \frac{C_{pen}\kappa}{h_K} \| v' \|^2_{0,\Gamma_K} \right. \nonumber \\
& & \left. -\int_{\Gamma_K} \frac{1}{2} a_n \left(v'\right)^2 + \int_{\Gamma_K} a_n^+ \left(v'\right)^2 - 2 \int_{\Gamma_K} \kappa \nabla v' \cdot \bm{n} v' \right) \nonumber \\
 & = & \sum_{K=1}^{n_{el}} \left( \left( \kappa + \kappa_{art} \right) | v' |^2_{1,\Omega_K} + \frac{C_{pen}\kappa}{h_K} \| v' \|^2_{0,\Gamma_K} + \frac{1}{2} \| a_n^{1/2} v' \|^2_{0,\Gamma_K} - 2 \int_{\Gamma_K} \kappa \nabla v' \cdot \bm{n} v' \right). \nonumber \\
\end{eqnarray}
For the bilinear forms coupling the coarse-scales and fine-scales, we evaluate:
\begin{eqnarray}
\bar{C}(v',\bar{v}) + C'(\bar{v},v') & = & -\sum_{K=1}^{n_{el}} \int_{\Omega_K} v' \nabla \cdot \left( \kappa \nabla \bar{v} \right) \nonumber \\
& = & \sum_{K=1}^{n_{el}} \left( \int_{\Omega_K} \kappa \nabla \bar{v} \cdot \nabla v' - \int_{\Gamma_K} \kappa \nabla \bar{v} \cdot \bm{n} v' \right).
\end{eqnarray}
To continue, we recognize that, for every element $\Omega_K \in \mathcal{M}$:
\begin{eqnarray}
\left|\int_{\Gamma_K} \kappa \nabla v' \cdot \bm{n} v' \right| \leq \frac{1}{2} \left( \frac{h_K}{4 C_{trace}} \kappa \| \nabla v' \cdot \bm{n} \|^2_{0,\Gamma_K} + \frac{4 C_{trace}}{h_K} \kappa \| v' \|^2_{0,\Gamma_K}  \right) \nonumber
\end{eqnarray}
by Young's inequality, and:
\begin{eqnarray}
\left|\int_{\Gamma_K} \kappa \nabla v' \cdot \bm{n} v' \right| \leq \frac{1}{2} \left( \frac{1}{4} \kappa | v' |^2_{1,\Omega_K} + \frac{4 C_{trace}}{h_K} \kappa \| v' \|^2_{0,\Gamma_K} \right)
\end{eqnarray}
by the trace inequality.  By a similar argument, we have:
\begin{eqnarray}
\left|\int_{\Gamma_K} \kappa \nabla \bar{v} \cdot \bm{n} v' \right| \leq \frac{1}{2} \left( \frac{1}{4} \kappa | \bar{v} |^2_{1,\Omega_K} + \frac{4 C_{trace}}{h_K} \kappa \| v' \|^2_{0,\Gamma_K} \right)
\end{eqnarray}
for every element $\Omega_K \in \mathcal{M}$.  Finally, by the triangle inequality, we have:
\begin{eqnarray}
\left| \int_{\Omega_K} \kappa \nabla \bar{v} \cdot \nabla v' \right| \leq \frac{1}{2} \left( \kappa | \bar{v} |^2_{1,\Omega_K} + \kappa | v' |^2_{1,\Omega_K} \right) \label{eq:th1_end}
\end{eqnarray}
for every element $\Omega_K \in \mathcal{M}$.  Collecting our results contained in \eqref{eq:th1_begin}-\eqref{eq:th1_end}, we obtain:
\begin{eqnarray}
\bar{B}(\bar{v},\bar{v}) & = & \kappa | \bar{v} |^2_{1,\Omega} \\
B'(v',v') & \geq & \sum_{K=1}^{n_{el}} \left( \left( \kappa + \kappa_{art} \right) | v' |^2_{1,\Omega_K} + \frac{C_{pen}\kappa}{h_K} \| v' \|^2_{0,\Gamma_K} + \frac{1}{2} \| a_n^{1/2} v' \|^2_{0,\Gamma_K} \right. \nonumber \\ && \left. - \frac{1}{4} \kappa | v' |^2_{1,\Omega_K} - \frac{4 C_{trace}}{h_K} \kappa \| v' \|^2_{0,\Gamma_K} \right) \\
\bar{C}(v',\bar{v}) + C'(\bar{v},v') & \geq & - \sum_{K=1}^{n_{el}} \frac{1}{2} \left( \kappa | \bar{v} |^2_{1,\Omega_K} + \kappa | v' |^2_{1,\Omega_K} \right. \nonumber \\ && \left. + \frac{1}{4} \kappa | \bar{v} |^2_{1,\Omega_K} + \frac{4 C_{trace}}{h_K} \kappa \| v' \|^2_{0,\Gamma_K} \right)
\end{eqnarray}
The desired expression follows by adding the above inequalities and invoking Assumption 1. \qed
\end{proof}

The coercivity result provided in Theorem 4.1 ensures that the group variable solution is unique, but it does not ensure that the corresponding methodology is stable.  In fact, note that in the limit $\kappa \rightarrow 0$, the coercivity result suggests that the methodology loses control of the coarse-scale solution entirely.  However, we are able to show that the methodology satisfies a different notion of stability, namely inf-sup stability, in this limit.  To proceed forward, let us define the norm:
\begin{eqnarray}
\| \textbf{V}^h \|^2_S &:=& \kappa | \bar{v} |^2_{1,\Omega} + \sum_{K=1}^{n_{el}} \left( \left( \kappa + \kappa_{art} \right) | v' |^2_{1,\Omega_K} + \frac{C_{pen}\kappa}{h_K} \| v' \|^2_{0,\Gamma_K} + \| a_n^{1/2} v' \|^2_{0,\Gamma_K} \right) \nonumber \\ && + \sum_{K=1}^{n_{el}} \| \tau_K^{1/2} \mathcal{P}'_K \left( \bm{a} \cdot \nabla \left(\bar{v} + v'\right) \right) \|^2_{0,\Omega_K}
\end{eqnarray}
for every $\textbf{V}^h \in \bm{\mathcal{V}}^h$ where $\mathcal{P}'_K$ is the $L^2$-projector onto the space of discontinuous subscales $\mathcal{V}'_{disc,K}$ on element $\Omega_K \in \mathcal{M}$ and $\tau_K$ is the element-wise constant:
\begin{equation}
\tau_K := C_{\tau} \left( \textup{max}\left\{ \frac{\| \bm{a} \|_{0,\infty,\Omega_K}}{h_K}, \frac{C_{inv} \kappa}{h^2_K} \right \}\right)^{-1} \nonumber
\end{equation}
where $C_{\tau} > 0$ is an arbitrary positive constant.  With the above norm defined, we have the following inf-sup stability result.

\begin{theorem}
The inf-sup stability result
\begin{equation}
\inf_{\textup{\textbf{U}}^h \in \bm{\mathcal{V}}^h} \sup_{\textup{\textbf{V}}^h \in \bm{\mathcal{V}}^h} \frac{\textup{\textbf{B}}(\textup{\textbf{U}}^h,\textup{\textbf{V}}^h)}{ \| \textup{\textbf{U}}^h \|_S \| \textup{\textbf{V}}^h \|_S } \geq \beta > 0 \nonumber
\end{equation}
holds provided Assumptions 1 and 2 are satisfied where $\beta$ is a positive constant independent of the problem parameters $\bm{a}$ and $\kappa$ and the mesh size $h = \max_{K} h_K$.
\end{theorem}

\begin{proof}
Let $\textbf{U}^h = (\bar{u},u')$ be an arbitrary member of $\bm{\mathcal{V}}^h$.  It suffices to show that there exists some $\textbf{V}^h = (\bar{v},v') \in \bm{\mathcal{V}}^h$ such that $\textbf{B}(\textbf{U}^h,\textbf{V}^h) \geq \beta \| \textbf{U}^h \|_S \| \textbf{V}^h \|_S$.  In this direction, let $z' \in \mathcal{V}'_{disc}$ be defined such that $z'|_{\Omega_K} = \tau_K \mathcal{P}_K' \left( \bm{a} \cdot \nabla \left( \bar{u} + u' \right) \right)$.  Note that we have the following inequalities bounding the element-wise $L^2$-norm of $z'$:
\begin{eqnarray}
\| z' \|_{0,\Omega_K} &\leq& \tau_K \| \bm{a} \cdot \nabla \left( \bar{u} + u' \right) \|_{0,\Omega_K} \nonumber \\
&\leq& \tau_K \| \bm{a} \|_{0,\infty,\Omega_K} | \bar{u} + u' |_{1,\Omega_K} \nonumber \\
&\leq& \tau_K C^{1/2}_{inv} h_K^{-1} \| \bm{a} \|_{0,\infty,\Omega_K} \| \bar{u} + u' \|_{0,\Omega_K} \nonumber \\
&\leq& C_{\tau} C^{1/2}_{inv}  \| \bar{u} + u' \|_{0,\Omega_K} \nonumber \\
&\leq& C_{\tau} C^{1/2}_{inv}  \left( \| \bar{u} \|_{0,\Omega_K} +  \| u' \|_{0,\Omega_K} \right),
\end{eqnarray}
the element-wise $L^2$-norm of the gradient of $z'$:
\begin{eqnarray}
\kappa^{1/2} | z' |_{1,\Omega_K} &\leq& \kappa^{1/2} C^{1/2}_{inv} h_K^{-1} \| z' \|_{0,\Omega_K} \nonumber \\
&=& \tau^{1/2}_K \kappa^{1/2} C^{1/2}_{inv} h_K^{-1} \| \tau^{-1/2}_K z' \|_{0,\Omega_K} \nonumber \\
&\leq& C^{1/2}_{\tau} \| \tau^{1/2}_K \mathcal{P}_K' \left( \bm{a} \cdot \nabla \left( \bar{u} + u' \right) \right) \|_{0,\Omega_K},
\end{eqnarray}
again the element-wise $L^2$-norm of the gradient of $z'$:
\begin{eqnarray}
\kappa_{art}^{1/2} | z' |_{1,\Omega_K} &\leq& \kappa_{art}^{1/2} C^{1/2}_{inv} h_K^{-1} \| z' \|_{0,\Omega_K} \nonumber \\
&=& \tau^{1/2}_K \kappa_{art}^{1/2} C^{1/2}_{inv} h_K^{-1} \| \tau^{-1/2}_K z' \|_{0,\Omega_K} \nonumber \\
&=& \tau_K C^{-1/2}_{\tau} C^{1/2}_{art} \| \bm{a} \|_{0,\infty,\Omega_K} C^{1/2}_{inv} h_K^{-1} \| \tau^{-1/2}_K z' \|_{0,\Omega_K} \nonumber \\
&\leq& C^{1/2}_{\tau} C^{1/2}_{art} C^{1/2}_{inv} \| \tau^{1/2}_K \mathcal{P}_K' \left( \bm{a} \cdot \nabla \left( \bar{u} + u' \right) \right) \|_{0,\Omega_K},
\end{eqnarray}
the element-wise $L^2$-norm of the trace of $z'$:
\begin{eqnarray}
\| a^{1/2}_n z'\|_{0,\Gamma_K} &\leq&\| \bm{a} \|^{1/2}_{0,\infty,\Omega_K} \| z' \|_{0,\Gamma_K} \nonumber \\
&\leq& C^{1/2}_{trace} h_K^{-1/2} \| \bm{a} \|^{1/2}_{0,\infty,\Omega_K} \| \tau_K \mathcal{P}_K' \left( \bm{a} \cdot \nabla \left( \bar{u} + u' \right) \right) \|_{0,\Omega_K} \nonumber \\
&\leq& \tau^{1/2}_K C^{1/2}_{trace} h_K^{-1/2} \| \bm{a} \|^{1/2}_{0,\infty,\Omega_K} \| \tau^{1/2}_K \mathcal{P}_K' \left( \bm{a} \cdot \nabla \left( \bar{u} + u' \right) \right) \|_{0,\Omega_K} \nonumber \\
&\leq& C^{1/2}_{\tau} C^{1/2}_{trace} \| \tau^{1/2}_K \mathcal{P}_K' \left( \bm{a} \cdot \nabla \left( \bar{u} + u' \right) \right) \|_{0,\Omega_K},
\end{eqnarray}
again the element-wise $L^2$-norm of the trace of $z'$:
\begin{eqnarray}
C^{1/2}_{pen} h_K^{-1/2} \kappa^{1/2} \| z' \|_{0,\Gamma_K} &\leq& C^{1/2}_{pen} C^{1/2}_{trace} h_K^{-1} \kappa^{1/2} \| z' \|_{0,\Omega_K} \nonumber \\
&=& \tau_K^{1/2} C^{1/2}_{pen} C^{1/2}_{trace} h_K^{-1} \kappa^{1/2} \| \tau^{-1/2}_K z' \|_{0,\Omega_K} \nonumber \\
&\leq& C^{1/2}_{\tau} C^{1/2}_{pen} C^{1/2}_{trace} C^{-1/2}_{inv} \| \tau^{1/2}_K \mathcal{P}_K' \left( \bm{a} \cdot \nabla \left( \bar{u} + u' \right) \right) \|_{0,\Omega_K}
\end{eqnarray}
and the element-wise $L^2$-norm of the normal derivative trace of $z'$:
\begin{eqnarray}
h_K^{1/2} \kappa^{1/2} \| \nabla z' \cdot \bm{n} \|_{0,\Gamma_K} &\leq& C^{1/2}_{trace} \kappa^{1/2} | z' |_{1,\Omega_K} \nonumber \\
&\leq&  C^{1/2}_{inv} C^{1/2}_{trace} h_K^{-1} \kappa^{1/2} \| z' \|_{0,\Omega_K} \nonumber \\
&=& \tau_K^{1/2} C^{1/2}_{inv} C^{1/2}_{trace} h_K^{-1} \kappa^{1/2} \| \tau^{-1/2} z' \|_{0,\Omega_K} \nonumber \\
&\leq& C^{1/2}_{\tau} C^{1/2}_{trace} \| \tau^{1/2}_K \mathcal{P}_K' \left( \bm{a} \cdot \nabla \left( \bar{u} + u' \right) \right) \|_{0,\Omega_K}.
\end{eqnarray}
We need one more inequality for $z'$.  Notably, observe:
\begin{eqnarray}
\| \tau_K^{1/2} \mathcal{P}'_K \left( \bm{a} \cdot \nabla z' \right) \|_{0,\Omega_K} & \leq &  \tau_K^{1/2} \| \bm{a} \|_{0,\infty,\Omega_K} | z' |_{1,\Omega_K} \nonumber \\
& \leq & \tau^{1/2}_{K} C^{1/2}_{inv} h_K^{-1} \| \bm{a} \|_{0,\infty,\Omega_K} \| z' \|_{0,\Omega_K} \nonumber \\
& = & \tau^{1/2}_{K} C^{1/2}_{inv} h_K^{-1} \| \bm{a} \|_{0,\infty,\Omega_K} \| \tau_K \mathcal{P}'_K \left( \bm{a} \cdot \nabla \left( \bar{u} + u' \right) \right) \|_{0,\Omega_K} \nonumber \\
& \leq & \tau_{K} C^{1/2}_{inv} h_K^{-1} \| \bm{a} \|_{0,\infty,\Omega_K} \| \tau^{1/2}_K \mathcal{P}'_K \left( \bm{a} \cdot \nabla \left( \bar{u} + u' \right) \right) \|_{0,\Omega_K} \nonumber \\
& \leq & C_{\tau} C^{1/2}_{inv}  \| \tau_K \mathcal{P}'_K \left( \bm{a} \cdot \nabla \left( \bar{u} + u' \right) \right) \|_{0,\Omega_K}.
\end{eqnarray}
Let us define $\textbf{Z}^h = (0,z')$.  Note immediately that as a consequence of the above inequalities, the following inequality holds:
\begin{equation}
\| \textbf{Z}^h \|_S \leq C_z \| \textbf{U}^h \|_S
\end{equation}
where:
\begin{equation}
C_z = \sqrt{ C_{\tau} \left( 1 + C_{trace} + C_{trace} C_{pen} C^{-1}_{inv} + C_{art} C_{inv} + C_{\tau} C_{inv} \right) }.
\end{equation}
We now seek a lower bound for the quantity $\textbf{B}(\textbf{U}^h,\textbf{Z}^h)$.  We observe that:
\begin{eqnarray}
\textbf{B}(\textbf{U}^h,\textbf{Z}^h) &=& B'(u',z') + C'(\bar{u},z') \nonumber \\
&=& \sum_{K=1}^{n_{el}} \left( -\int_{\Omega_K} u' \bm{a} \cdot \nabla z' + \int_{\Omega_K} \left( \kappa + \kappa_{art} \right) \nabla u' \cdot \nabla z' + \int_{\Gamma_K} u' z' a_n^+ \right. \nonumber \\
& & - \int_{\Gamma_K} \kappa \nabla u' \cdot \bm{n} z' - \int_{\Gamma_K} u' \kappa \nabla z' \cdot \bm{n} + \int_{\Gamma_K} \frac{C_{pen} \kappa}{h_K} u' z' \nonumber \\
&& \left. + \int_{\Omega_K} z' \left( \bm{a} \cdot \nabla \bar{u} - \nabla \cdot \left( \kappa \nabla \bar{u} \right) \right) \right) \nonumber \\
&=& \sum_{K=1}^{n_{el}} \left( \int_{\Omega_K} \bm{a} \cdot \nabla \left( \bar{u} + u' \right) z' + \int_{\Omega_K} \left( \kappa + \kappa_{art} \right) \nabla u' \cdot \nabla z' \right. \nonumber \\
&& - \int_{\Gamma_K} u' z' a_n^- - \int_{\Gamma_K} \kappa \nabla u' \cdot \bm{n} z' - \int_{\Gamma_K} u' \kappa \nabla z' \cdot \bm{n} + \int_{\Gamma_K} \frac{C_{pen} \kappa}{h_K} u' z' \nonumber \\
&& \left. - \int_{\Omega_K} z' \left(\nabla \cdot \left( \kappa \nabla \bar{u} \right) \right) \right).
\end{eqnarray}
We deal with each of the expressions appearing above one-by-one.  First, we note that, by definition:
\begin{eqnarray}
\int_{\Omega_K} \bm{a} \cdot \nabla \left( \bar{u} + u' \right) z' &=& \tau_K \int_{\Omega_K} \bm{a} \cdot \nabla \left( \bar{u} + u' \right) \mathcal{P}'_K \left( \bm{a} \cdot \nabla \left( \bar{u} + u' \right) \right) \nonumber \\
&=& \| \tau_K^{1/2} \mathcal{P}'_K \left( \bm{a} \cdot \nabla \left( \bar{u} + u' \right) \right) \|^2_{0,\Omega_K}.
\end{eqnarray}
For the next term, we have:
\begin{eqnarray}
\left| \int_{\Omega_K} \kappa \nabla u' \cdot \nabla z' \right| & \leq & \kappa | u' |_{1,\Omega_K} | z' |_{1,\Omega_K} \nonumber \\
& \leq & C^{1/2}_{\tau} \kappa^{1/2} | u' |_{1,\Omega_K} \| \tau^{1/2}_K \mathcal{P}_K' \left( \bm{a} \cdot \nabla \left( \bar{u} + u' \right) \right) \|_{0,\Omega_K} \nonumber \\
& \leq & C^{1/2}_{\tau} \left( \frac{\kappa}{2\gamma_1} | u' |_{1,\Omega_K}^2 + \frac{\gamma_1}{2} \| \tau^{1/2}_K \mathcal{P}_K' \left( \bm{a} \cdot \nabla \left( \bar{u} + u' \right) \right) \|_{0,\Omega_K}^2 \right)
\end{eqnarray}
where $\gamma_1 > 0$ is an arbitrarily chosen positive number.  In analogous fashion, we have:
\begin{eqnarray}
\left| \int_{\Omega_K} \kappa_{art} \nabla u' \cdot \nabla z' \right| & \leq & C^{1/2}_{\tau} C^{1/2}_{art} C^{1/2}_{inv} \left( \frac{\kappa_{art}}{2\gamma_2} | u' |_{1,\Omega_K}^2 + \frac{\gamma_2}{2} \| \tau^{1/2}_K \mathcal{P}_K' \left( \bm{a} \cdot \nabla \left( \bar{u} + u' \right) \right) \|_{0,\Omega_K}^2 \right) \nonumber \\
\end{eqnarray}
\begin{eqnarray}
\left| \int_{\Gamma_K} u' z' a_n^- \right| & \leq & C^{1/2}_{\tau} C^{1/2}_{trace} \left( \frac{1}{2\gamma_3} \| a_n^{1/2} u' \|_{0,\Gamma_K}^2 + \frac{\gamma_3}{2} \| \tau^{1/2}_K \mathcal{P}_K' \left( \bm{a} \cdot \nabla \left( \bar{u} + u' \right) \right) \|_{0,\Omega_K}^2 \right) \nonumber \\
\end{eqnarray}
\begin{eqnarray}
\left| \int_{\Gamma_K} \frac{C_{pen} \kappa}{h_K} u' z' \right| & \leq & C^{1/2}_{\tau} C^{1/2}_{pen} C^{-1/2}_{trace} C^{-1/2}_{inv} \left( \frac{C_{pen} \kappa}{2\gamma_4 h_K} \| u' \|_{0,\Gamma_K}^2 + \frac{\gamma_4}{2} \| \tau^{1/2}_K \mathcal{P}_K' \left( \bm{a} \cdot \nabla \left( \bar{u} + u' \right) \right) \|_{0,\Omega_K}^2 \right) \nonumber \\
\end{eqnarray}
where again $\gamma_2, \gamma_3, \gamma_4 > 0$ are arbitrarily chosen positive numbers.  The remaining terms require slightly more care.  For the first remaining term, we have:
\begin{eqnarray}
\left| \int_{\Gamma_K} \kappa \nabla u' \cdot \bm{n} z' \right| & \leq & \left( C_{pen}^{-1/2} h_K^{1/2} \kappa^{1/2} \| \nabla u' \cdot \bm{n} \|_{0,\Gamma_K} \right) \left( C_{pen}^{1/2} \kappa^{1/2} h_K^{-1/2} \| z' \|_{0,\Gamma_K} \right) \nonumber \\
& \leq & \left( C_{trace}^{1/2} C_{pen}^{-1/2} \kappa^{1/2} | u' |_{1,\Omega_K} \right) \left( C_{pen}^{1/2} \kappa^{1/2} h_K^{-1/2} \| z' \|_{0,\Gamma_K} \right) \nonumber \\
& \leq & \left( \kappa^{1/2} | u' |_{1,\Omega_K} \right) \left( C_{pen}^{1/2} \kappa^{1/2} h_K^{-1/2} \| z' \|_{0,\Gamma_K} \right) \nonumber \\
& \leq & C^{1/2}_{\tau} C^{1/2}_{pen} C^{-1/2}_{trace} C^{-1/2}_{inv} \left( \frac{\kappa}{2\gamma_5} | u' |_{1,\Omega_K}^2 + \frac{\gamma_5}{2} \| \tau^{1/2}_K \mathcal{P}_K' \left( \bm{a} \cdot \nabla \left( \bar{u} + u' \right) \right) \|_{0,\Omega_K}^2 \right) \nonumber \\
\end{eqnarray}
where $\gamma_5 > 0$ is another arbitrarily chosen positive number.  For the second remaining term, we have:
\begin{eqnarray}
\left| \int_{\Gamma_K} u' \kappa \nabla z' \cdot \bm{n} \right| & \leq & \left( C_{pen}^{1/2} \kappa^{1/2} h_K^{-1/2} \| u' \|_{0,\Gamma_K} \right) \left( C_{pen}^{-1/2} h_K^{1/2} \kappa^{1/2} \| \nabla z' \cdot \bm{n} \|_{0,\Gamma_K} \right) \nonumber \\
& = & C_{\tau}^{1/2} \left( \frac{C_{pen} \kappa}{2\gamma_6 h_K} \| u' \|_{0,\Gamma_K}^2 + \frac{\gamma_6}{2} \| \tau^{1/2}_K \mathcal{P}_K' \left( \bm{a} \cdot \nabla \left( \bar{u} + u' \right) \right) \|_{0,\Omega_K}^2 \right) \nonumber \\
\end{eqnarray}
where $\gamma_6 > 0$ is yet another arbitrarily chosen positive number.  For the last remaining term, we integrate by parts, resulting in:
\begin{eqnarray}
\int_{\Omega_K} z' \left(\nabla \cdot \left( \kappa \nabla \bar{u} \right) \right) = - \int_{\Omega_K} \kappa \nabla \bar{u} \cdot \nabla z' + \int_{\Gamma_K} \kappa \nabla \bar{u} \cdot \bm{n} z'
\end{eqnarray}
The two terms on the right-hand-side are then easily bound as before, yielding the inequality:
\begin{eqnarray}
\left| \int_{\Omega_K} z' \left(\nabla \cdot \left( \kappa \nabla \bar{u} \right) \right) \right| & \leq & C^{1/2}_{\tau} \left( \frac{\kappa}{2\gamma_7} | \bar{u} |_{1,\Omega_K}^2 + \frac{\gamma_7}{2} \| \tau^{1/2}_K \mathcal{P}_K' \left( \bm{a} \cdot \nabla \left( \bar{u} + u' \right) \right) \|_{0,\Omega_K}^2 \right) \nonumber \\
&& + C^{1/2}_{\tau} C^{1/2}_{pen} C^{-1/2}_{trace} C^{-1/2}_{inv} \left( \frac{\kappa}{2\gamma_8} | \bar{u} |_{1,\Omega_K}^2 + \frac{\gamma_8}{2} \| \tau^{1/2}_K \mathcal{P}_K' \left( \bm{a} \cdot \nabla \left( \bar{u} + u' \right) \right) \|_{0,\Omega_K}^2 \right) \nonumber \\
\end{eqnarray}
where $\gamma_7 > 0, \gamma_8 > 0$ are two final arbitrarily chosen positive numbers.  Collecting all of the above inequalities, we obtain the composite inequality:
\begin{eqnarray}
\textbf{B}(\textbf{U}^h,\textbf{Z}^h) \geq C_1 \sum_{K=1}^{n_{el}} \| \tau_K^{1/2} \mathcal{P}'_K \left( \bm{a} \cdot \nabla \left(\bar{v} + v'\right) \right) \|^2_{0,\Omega_K} - C_2 \kappa | \bar{v} |^2_{1,\Omega} \nonumber \\ - \sum_{K=1}^{n_{el}} \left( \left( C_3 \kappa + C_4 \kappa_{art} \right) | v' |^2_{1,\Omega_K} + C_5 \frac{C_{pen}\kappa}{h_K} \| v' \|^2_{0,\Gamma_K} + C_6 \| a_n^{1/2} v' \|^2_{0,\Gamma_K} \right) \label{eqn:composite_inequality}
\end{eqnarray}
wherein:
\begin{align}
C_1 &= 1 - \frac{C_{\tau}^{1/2}}{2} \left( \left( \gamma_1 + \gamma_6 + \gamma_7 \right) + C_{art}^{1/2} C_{inv}^{1/2} \gamma_2 + C_{trace}^{1/2} \gamma_3 + C_{pen}^{1/2} C_{trace}^{-1/2} C_{inv}^{-1/2} \left( \gamma_4 + \gamma_5 + \gamma_8 \right) \right) \nonumber
\end{align}
\begin{align}
C_2 &= \frac{C_{\tau}^{1/2}}{2\gamma_2} + \frac{C_{\tau}^{1/2} C_{pen}^{1/2} C_{trace}^{-1/2} C_{inv}^{-1/2}}{2\gamma_8} \nonumber \\
C_3 &= \frac{C_{\tau}^{1/2}}{2\gamma_1} + \frac{C_{\tau}^{1/2} C_{pen}^{1/2} C_{trace}^{-1/2} C_{inv}^{-1/2}}{2\gamma_5} \nonumber \\
C_4 &= \frac{C_{\tau}^{1/2}C_{art}^{1/2}C_{inv}^{1/2}}{2\gamma_3} \nonumber \\
C_5 &= \frac{C_{\tau}^{1/2} C_{pen}^{1/2} C_{trace}^{-1/2} C_{inv}^{-1/2}}{2\gamma_4}+ \frac{C_{\tau}^{1/2}}{2\gamma_6}\nonumber \\
C_6 &=  \frac{C_{\tau}^{1/2}C_{trace}^{1/2}}{2\gamma_3}. \nonumber
\end{align}
We now assume that $\gamma_1$ through $\gamma_8$ are chosen sufficiently small to guarantee that $C_1 > 0$.  Note that we can choose such constants independent of the problem parameters and the mesh size.  This choice in turn defines the constants $C_2$ through $C_6$.  We are now in a position to define a suitable group test function $\textbf{V}^h = (\bar{v},v') \in \bm{\mathcal{V}}^h$.  Namely, we select $\bar{v} = \bar{u}$ and $v' = u' + C_{lin} z'$ where:
\begin{eqnarray}
C_{lin} = \frac{1}{8} \min_{2 \leq i \leq 6}\left\{\frac{1}{C_i}\right\}
\end{eqnarray}
Then, by Theorem \ref{theorem1} and (\ref{eqn:composite_inequality}), it follows that:
\begin{eqnarray}
\textbf{B}(\textbf{U}^h,\textbf{V}^h) \geq C_{bound} \| \textbf{U}^h \|^2_S
\end{eqnarray}
with:
\begin{eqnarray}
C_{bound} = \min \left\{ \frac{1}{8}, C_1 C_{lin} \right\}
\end{eqnarray}
and:
\begin{eqnarray}
\| \textbf{V}^h \|_S \leq C_v \| \textbf{U}^h \|_S
\end{eqnarray}
with $C_v = 1 + C_{lin} C_z$.  Thus the desired condition holds with $\beta = C_{bound}/C_v$ independent of the problem parameters $\bm{a}$ and $\kappa$ and the mesh size $h$. \qed
\end{proof}

We now introduce one more assumption.  This assumption guarantees that our methodology is at least as stable as the SUPG method for steady scalar transport, as is shown in Corollary 4.3.\\

\noindent \textbf{Assumption 3:} \textit{The following inequality holds for each element $\Omega_K \in \mathcal{M}$:
\begin{equation}
\| \tau_K^{1/2} \bm{a} \cdot \nabla \left(\bar{v} + v'\right) \|^2_{0,\Omega_K} \leq C_{SUPG} \left( \kappa_{art} | v' |^2_{1,\Omega_K} + \| \tau_K^{1/2} \mathcal{P}'_K \left( \bm{a} \cdot \nabla \left(\bar{v} + v'\right) \right) \|^2_{0,\Omega_K} \right) \nonumber
\end{equation}
where $C_{SUPG} > 0$ is a positive constant independent of the problem parameters $\bm{a}$ and $\kappa$ and the mesh size $h$.}

\begin{corollary}
Provided Assumptions 1, 2, and 3 hold, the following inf-sup stability result is satisfied:
\begin{equation}
\inf_{\textup{\textbf{U}}^h \in \bm{\mathcal{V}}^h} \sup_{\textup{\textbf{V}}^h \in \bm{\mathcal{V}}^h} \frac{\textup{\textbf{B}}(\textup{\textbf{U}}^h,\textup{\textbf{V}}^h)}{ \| \textup{\textbf{U}}^h \|_{SUPG} \| \textup{\textbf{V}}^h \|_{SUPG} } \geq \beta_{SUPG} > 0 \nonumber
\end{equation}
where:
\begin{eqnarray}
\| \textbf{\textup{V}}^h \|^2_{SUPG} &=& \kappa | \bar{v} |^2_{1,\Omega} + \sum_{K=1}^{n_{el}} \left( \left( \kappa + \kappa_{art} \right) | v' |^2_{1,\Omega_K} + \frac{C_{pen}\kappa}{h_K} \| v' \|^2_{0,\Gamma_K} + \| a_n^{1/2} v' \|^2_{0,\Gamma_K} \right) \nonumber \\ && + \sum_{K=1}^{n_{el}} \| \tau_K^{1/2} \bm{a} \cdot \nabla \left(\bar{v} + v'\right) \|^2_{0,\Omega_K} \nonumber
\end{eqnarray}
for every $\textbf{\textup{V}}^h \in \bm{\mathcal{V}}^h$ and $\beta_{SUPG}$ is a positive constant independent of the problem parameters $\bm{a}$ and $\kappa$ and the mesh size $h$.
\end{corollary}

There remains the question of whether or not we can expect Assumption 3 to be satisfied for a given finite element discretization.  The following lemma demonstrates that Assumption 3 is indeed satisfied if the discontinuous subscale solution space is sufficiently rich and the artificial diffusivity is chosen in an intelligent manner.

\begin{lemma}
Assumption 3 is satisfied provided that one of the following two conditions is satisfied:
\vspace{3pt}

\textbf{\textup{C1:}} It holds that $\bm{a} \cdot \nabla \left(\bar{v} + v' \right) \in \mathcal{V}'_{disc}$ for every $(\bar{v},v') \in \bm{\mathcal{V}}^h$.
\vspace{3pt}

\textbf{\textup{C2:}} It holds that $\bm{a} \cdot \nabla \bar{v} \in \mathcal{V}'_{disc}$ for every $\bar{v} \in \overline{\mathcal{V}}$ and Assumption 2 is satisfied with $C_{art} > 0$.
\end{lemma}

\begin{proof}
We write:
\begin{equation}
\| \tau_K^{1/2} \bm{a} \cdot \nabla \left(\bar{v} + v'\right) \|^2_{0,\Omega_K} = \| \tau_K^{1/2} \mathcal{P}'_K \left( \bm{a} \cdot \nabla \left(\bar{v} + v'\right) \right) \|^2_{0,\Omega_K} + \| \tau_K^{1/2} \left( \mathcal{I} - \mathcal{P}'_K \right) \left( \bm{a} \cdot \nabla \left(\bar{v} + v'\right) \right) \|^2_{0,\Omega_K} \nonumber
\end{equation}
where $\mathcal{I}$ is the identity operator.  If Condition C1 holds, it follows that:
\begin{equation}
\| \tau_K^{1/2} \left( \mathcal{I} - \mathcal{P}'_K \right) \left( \bm{a} \cdot \nabla \left(\bar{v} + v'\right) \right) \|^2_{0,\Omega_K} = 0 \nonumber
\end{equation}
and consequently the lemma is satisfied with $C_{SUPG} = 1$.  If Condition C2 holds, we instead have:
\begin{align}
\| \tau_K^{1/2} \left( \mathcal{I} - \mathcal{P}'_K \right) \left( \bm{a} \cdot \nabla \left(\bar{v} + v'\right) \right) \|^2_{0,\Omega_K} &= \| \tau_K^{1/2} \left( \mathcal{I} - \mathcal{P}'_K \right) \left( \bm{a} \cdot \nabla v' \right) \|^2_{0,\Omega_K} \nonumber \\
&\leq \tau_K \| \bm{a} \cdot \nabla v' \|^2_{0,\Omega_K} \nonumber \\
&\leq \tau_K \| \bm{a} \|_{0,\infty,\Omega_K} | v' |_{1,\Omega_K}. \nonumber
\end{align}
As Assumption 2 is satisfied with $C_{art} > 0$, it follows that:
\begin{equation}
\| \tau_K^{1/2} \left( \mathcal{I} - \mathcal{P}'_K \right) \left( \bm{a} \cdot \nabla \left(\bar{v} + v'\right) \right) \|^2_{0,\Omega_K} \leq \frac{C_{\tau}}{C_{art}} \kappa_{art} | v' |_{1,\Omega_K} \nonumber
\end{equation}
and consequently the lemma is satisfied with $C_{SUPG} = \max\left\{1,\frac{C_{\tau}}{C_{art}}\right\}$.
\qed
\end{proof}

Suppose that the imposed velocity field is a polynomial function of degree $q$ over each element $\Omega_K \in \mathcal{M}$.  Then, for the case of a finite element mesh of simplices, we see that $\bm{a} \cdot \nabla \bar{v}$ is a polynomial function of degree $p + q - 1$ over each element and $\bm{a} \cdot \nabla v'$ is a polynomial function of degree $p_f + q - 1$ over each element.  Consequently, if $p_f \geq p + q - 1$ and $q \leq 1$, then Condition C1 of Lemma 4.4 is satisfied, and if only $p_f \geq p + q - 1$ is satisfied but $C_{art}$ is chosen to be a positive number, then Condition C2 of Lemma 4.4 is satisfied.  For a continuous piecewise linear finite element discretization with discontinuous piecewise linear subscales, we see that Condition C1 of Lemma 4.4 is satisfied for a piecewise linear velocity field.  Alternately, for a continuous piecewise quadratic finite element discretization with discontinuous piecewise linear subscales, we see that Condition C1 of Lemma 4.4 is satisfied for a piecewise constant velocity field.  By enriching the subscale space to discontinuous piecewise quadratic subscales, we find Condition C1 of Lemma 4.4 is satisfied again for a piecewise linear velocity field.

Lemma 4.4 provides a general guideline for how to choose the polynomial degree of the subscale space.  Nonetheless, we have observed our methodology often returns accurate and stable results even when the conditions of the lemma are not satisfied.  In particular, we have observed our methodology is stable if we employ smooth splines for our coarse-scale solution and discontinuous piecewise bi-linear finite elements for our subscale solution.  We anticipate that Lemma 4.4 holds in this case, though such an analysis is beyond the scope of the current work.

Our final theorem demonstrates that our method exhibits optimal convergence rates with respect to the SUPG norm provided Assumptions 1, 2, and 3 hold.

\begin{theorem}
Suppose that the exact solution satisfies the smoothness condition $u \in H^{p+1}(\Omega)$ and Assumptions 1, 2, and 3 hold.  Then, the error $\textup{\textbf{E}} = (u-\bar{u},-u')$ satisfies the \textit{a priori} estimate:
\begin{equation}
\| \textup{\textbf{E}} \|^2_{SUPG} \leq C_{a priori} \sum_{K=1}^{n_{el}} \left( \| \bm{a} \|_{0,\infty,\Omega_K} h^{2p+1}_K + \kappa h^{2p}_K \right) | u |^2_{p+1,\Omega_K}. \nonumber
\end{equation}
where $C_{a priori}$ is a positive constant independent of the problem parameters $\bm{a}$ and $\kappa$ and the mesh size $h$.
\end{theorem}

\begin{proof}
Let $\tilde{u} \in \overline{\mathcal{V}}$ be an interpolation function which we will define later, and let us split the error into two components, a method error defined as $\textbf{E}^h = (\tilde{u}-\bar{u},-u') = (e^h,e')$ and an interpolation error defined as $\boldsymbol{\eta} = (u-\tilde{u},0) = (\eta,0)$.  Our first objective is to bound the method error by the interpolation error.  By Corollary 4.3, we have that:
\begin{equation}
\beta_{SUPG} \| \textup{\textbf{E}}^h \|_{SUPG} \leq \sup_{\textup{\textbf{V}}^h \in \bm{\mathcal{V}}^h} \frac{\textup{\textbf{B}}(\textup{\textbf{E}}^h,\textup{\textbf{V}}^h)}{\| \textup{\textbf{V}}^h \|_{SUPG} }. \label{eq:apriori_1}
\end{equation}
Furthermore, it is readily shown that our method is consistent, that is, $\textup{\textbf{B}}(\textup{\textbf{E}},\textup{\textbf{V}}^h) = 0$, so we further have that:
\begin{equation}
\beta_{SUPG} \| \textup{\textbf{E}}^h \|_{SUPG} \leq \sup_{\textup{\textbf{V}}^h \in \bm{\mathcal{V}}^h} \frac{-\textup{\textbf{B}}(\boldsymbol{\eta},\textup{\textbf{V}}^h)}{\| \textup{\textbf{V}}^h \|_{SUPG} }. \label{eq:apriori_2}
\end{equation}
We now require a bound on the term $\textup{\textbf{B}}(\boldsymbol{\eta},\textup{\textbf{V}}^h)$.  Direct substitution results in:
\begin{equation}
\textup{\textbf{B}}(\boldsymbol{\eta},\textup{\textbf{V}}^h) = \bar{B}(\eta,\bar{v}) + C'(\eta,v'). \label{eq:apriori_3}
\end{equation}
We can bound the first term on the right hand side of \eqref{eq:apriori_3} as:
\begin{align}
|\bar{B}(\eta,\bar{v})| & \leq \left|\int_{\Omega} \eta \bm{a} \cdot \nabla \bar{v}\right| + \left|\int_{\Omega} \kappa \nabla \eta \cdot \nabla \bar{v}\right| \nonumber \\
& \leq \sum_{K=1}^{n_{el}} \left( \left\| \tau_K^{-1/2} \eta \right\|_{0,\Omega_K} \left\| \tau_K^{1/2} \bm{a} \cdot \nabla \bar{v} \right\|_{0,\Omega_K} + \kappa \left| \eta \right|_{1,\Omega_K} \left| \bar{v} \right|_{1,\Omega_K} \right). \label{eq:apriori_4}
\end{align}
The second term on the right hand side of \eqref{eq:apriori_3} requires more care.  We first integrate by parts to obtain:
\begin{align}
C'(\eta,v') &= \sum_{K=1}^{n_{el}} \int_{\Omega_K} v' \left( \bm{a} \cdot \nabla \eta  - \nabla \cdot \left( \kappa \nabla \eta \right) \right) \nonumber \\
&= \sum_{K=1}^{n_{el}} \left( -\int_{\Omega_K} \eta \bm{a} \cdot \nabla v' + \int_{\Gamma_K} \eta v' a_n  + \int_{\Omega_K} \kappa \nabla \eta \cdot \nabla v' - \int_{\Gamma_K} \kappa \nabla \eta \cdot \bm{n} v' \right). \label{eq:apriori_5}
\end{align}
To proceed, we require bounds for each of the four terms on the right hand side of \eqref{eq:apriori_5}.  The first and third terms are easily bounded as before, yielding:
\begin{align}
\left| \int_{\Omega_K} \eta \bm{a} \cdot \nabla v' \right| & \leq \left\| \tau_K^{-1/2} \eta \right\|_{0,\Omega_K} \left\| \tau_K^{1/2} \bm{a} \cdot \nabla v' \right\|_{0,\Omega_K} \label{eq:apriori_6} \\
\left| \int_{\Omega_K} \kappa \nabla \eta \cdot \nabla v' \right|  & \leq \kappa \left| \eta \right|_{1,\Omega_K} \left| v' \right|_{1,\Omega_K}. \label{eq:apriori_7}
\end{align}
Similarly, we can bound the second and fourth terms like:
\begin{align}
\left| \int_{\Gamma_K} \eta v' a_n \right| & \leq \left\| a_n^{1/2} \eta \right\|_{0,\Gamma_K} \left\| a_n^{1/2} v' \right\|_{0,\Gamma_K} \label{eq:apriori_8} \\
\left| \int_{\Gamma_K} \kappa \nabla \eta \cdot \bm{n} v' \right|  & \leq \kappa \left\| \nabla \eta \cdot \bm{n} \right\|_{0,\Gamma_K} \left\| v' \right\|_{0,\Gamma_K}. \label{eq:apriori_9}
\end{align}
Collecting \eqref{eq:apriori_3}-\eqref{eq:apriori_9} and applying the Cauchy-Schwarz inequality, we obtain the composite inequality:
\begin{align}
\frac{|\textup{\textbf{B}}(\boldsymbol{\eta},\textup{\textbf{V}}^h)|}{\| \textup{\textbf{V}}^h \|_{SUPG}} &\leq \left( \sum_{K=1}^{n_{el}} \left( 2 \left\| \tau_K^{-1/2} \eta \right\|^2_{0,\Omega_K} + 2 \kappa \left| \eta \right|^2_{1,\Omega_K} + \left\| a_n^{1/2} \eta \right\|^2_{0,\Gamma_K} + \frac{\kappa h_K}{C_{pen}} \left\| \nabla \eta \cdot \bm{n} \right\|^2_{0,\Gamma_K} \right) \right)^{1/2}. \nonumber
\end{align}
Combining the above expression with \eqref{eq:apriori_2}, we finally obtain the method error bound:
\begin{align}
\| \textup{\textbf{E}}^h \|^2_{SUPG} &\leq \beta^{-1/2}_{SUPG} \sum_{K=1}^{n_{el}} \left( 2 \left\| \tau_K^{-1/2} \eta \right\|^2_{0,\Omega_K} + 2 \kappa \left| \eta \right|^2_{1,\Omega_K} + \left\| a_n^{1/2} \eta \right\|^2_{0,\Gamma_K} + \frac{\kappa h_K}{C_{pen}} \left\| \nabla \eta \cdot \bm{n} \right\|^2_{0,\Gamma_K} \right). \label{eq:apriori_10}
\end{align}
Now suppose we have chosen the interpolant $\tilde{u}$ to be a ``best'' interpolant such that the following local interpolation estimates hold for every element $\Omega_K \in \mathcal{M}$:
\begin{align}
\left\| \eta \right\|^2_{0,\Omega_K} &\leq C_{shape} h_K^{2p+2} | u |^2_{p+1,\Omega_K} \nonumber \\
\left| \eta \right|^2_{1,\Omega_K} &\leq C_{shape} h_K^{2p} | u |^2_{p+1,\Omega_K} \nonumber \\
\left\| \eta \right\|^2_{0,\Gamma_K} &\leq C_{shape} h_K^{2p+1} | u |^2_{p+1,\Omega_K} \nonumber \\
\left\| \nabla \eta \cdot \bm{n} \right\|^2_{0,\Gamma_K} &\leq C_{shape} h_K^{2p-1} | u |^2_{p+1,\Omega_K} \nonumber
\end{align}
where $C_{shape} > 0$ is a positive constant independent of the mesh size but possibly dependent on the mesh regularity and coarse-scale polynomial degree $p$ \cite{Ciarlet02}.  It then follows that:
\begin{align}
\left\| \tau_K^{-1/2} \eta \right\|^2_{0,\Omega_K} &= \tau_K^{-1} \left\| \eta \right\|^2_{0,\Omega_K} \nonumber \\
&\leq C_{shape} C_{\tau}^{-1} \max\left\{ \frac{ \| \bm{a} \|_{0,\infty,\Omega_K}}{h_K}, \frac{C_{inv}\kappa}{h_K^2} \right\} h_K^{2p+2} | u |^2_{p+1,\Omega_K} \label{eq:apriori_11} \\
\kappa \left| \eta \right|^2_{1,\Omega_K} &\leq C_{shape} \kappa h_K^{2p} | u |^2_{p+1,\Omega_K} \label{eq:apriori_12} \\
\left\| a_n^{1/2} \eta \right\|^2_{0,\Gamma_K} &\leq \| \bm{a} \|_{0,\infty,\Omega_K} \left\| \eta \right\|^2_{0,\Gamma_K} \nonumber \\
&\leq C_{shape} \| \bm{a} \|_{0,\infty,\Omega_K} h_K^{2p+1} | u |^2_{p+1,\Omega_K} \label{eq:apriori_13} \\
\frac{\kappa h_K}{C_{pen}} \left\| \nabla \eta \cdot \bm{n} \right\|^2_{0,\Gamma_K} &\leq C_{shape} C_{pen}^{-1} \kappa h_K^{2p} | u |^2_{p+1,\Omega_K} \label{eq:apriori_14}
\end{align}
Combining \eqref{eq:apriori_11}-\eqref{eq:apriori_14} with the method error bound \eqref{eq:apriori_10} yields:
\begin{align}
\| \textup{\textbf{E}}^h \|^2_{SUPG} &\leq C_{method} \sum_{K=1}^{n_{el}} \left( \| \bm{a} \|_{0,\infty,\Omega_K} h^{2p+1}_K + \kappa h^{2p}_K \right) | u |^2_{p+1,\Omega_K} \label{eq:apriori_15}
\end{align}
where:
\begin{equation}
C_{method} = \beta_{SUPG}^{-1/2} C_{shape} \left( 2 C_{\tau}^{-1} + C_{pen}^{-1} + 2 \right). \nonumber
\end{equation}
The interpolation error may be bounded in a similar manner, resulting in:
\begin{align}
\| \boldsymbol{\eta} \|^2_{SUPG} &\leq C_{interpolation} \sum_{K=1}^{n_{el}} \left( \| \bm{a} \|_{0,\infty,\Omega_K} h^{2p+1}_K + \kappa h^{2p}_K \right) | u |^2_{p+1,\Omega_K} \label{eq:apriori_16}
\end{align}
where $C_{interpolation} > 0$ is a positive constant independent of the mesh size but possibly dependent on the mesh regularity and coarse-scale polynomial degree $p$.  The desired result follows by combining \eqref{eq:apriori_15} and \eqref{eq:apriori_16} with the decomposition $\textbf{E} = \textbf{E}^h + \boldsymbol{\eta}$.
\qed
\end{proof}

Note remarkably that the error estimate characterized by Theorem 5.5 is completely independent of the subscale polynomial degree.  Consequently, the subscales act purely to stabilize the coarse-scale solution and have virtually no impact on solution accuracy.  This is one of the main messages of the current work: it is possible to stabilize a high-order numerical method with a low-order, but stable, numerical method.

It should also be remarked that while the constants appearing in all of the above estimates are independent of the problem parameters $\bm{a}$ and $\kappa$ and the mesh size $h$, they are possibly dependent on mesh regularity, the coarse-scale polynomial degree $p$, and the subscale polynomial degree $p_f$.  A more delicate analysis is required to obtain dependencies with respect to these variables.  This would be useful, for instance, in the context of a boundary layer application wherein a highly skewed mesh should be employed near solid boundaries.

\begin{figure}[b!]
  \begin{center}
      \subfigure[L2-norm of error in coarse scales]{\includegraphics[width=8cm]{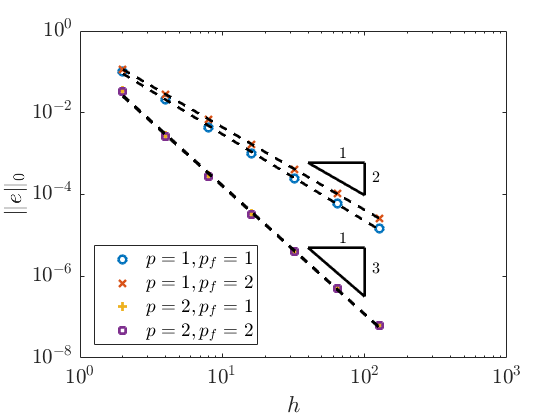}}
      \subfigure[L2-norm of fine scales]{\includegraphics[width=8cm]{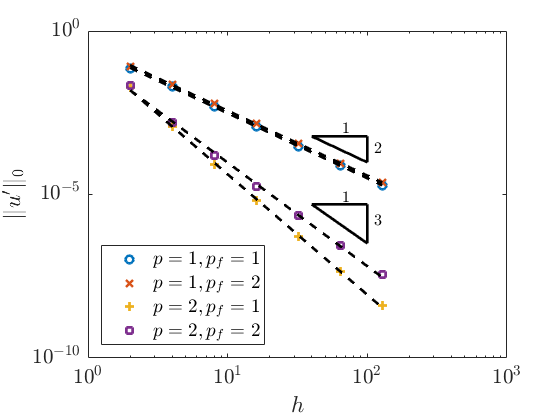}}
    \caption{Convergence for steady manufactured solution for varying polynomial degrees.}
    \label{fig:convergence}
  \end{center}
\end{figure}

\section{Numerical Results}

We finish this paper by applying our methodology to a collection of two-dimensional steady and unsteady transport problems.  All of the following results were obtained by applying the method of discontinuous subscales to isogeometric Non-Uniform Rational B-splines (NURBS) discretizations on uniform grids \cite{Hughes05}.  These discretizations employ $C^{p-1}$-continuous piecewise tensor-product polynomials or rational functions of degree $p$.  For $p = 1$, isogeometric NURBS discretizations correspond to standard bilinear finite elements, while for $p > 1$, isogeometric NURBS discretizations exhibit enhanced continuity as compared to standard tensor-product finite elements.  For more information on NURBS-based isogeometric analysis, refer to \cite{Cottrell09}.  Throughout, we used the following values for $C_{pen}$ and $\kappa_{art}$:
\begin{equation}\label{eq:C_ideal}
C_{pen} = 4\left( p_f + 1 \right)^2
\end{equation}
and:
\begin{equation}\label{eq:kart_ideal}
\kappa_{art} = \frac{h_K \left| \bm{a} \right|}{6 p_f}
\end{equation}
where $p_f$ is the subscale polynomial degree and $h_K$ is the mesh size for a given element $\Omega_K \in \mathcal{M}$.  These are inspired by theoretical estimates for the trace and inverse constants associated with a discontinuous subscale discretization \cite{Harari92,Warburton03}. Note that these values are simpler than the ones proposed in the previous section, but we have found that they still yield stable and accurate results.

\subsection{Steady manufactured solution}

We begin by considering a simple manufactured solution $u\left(x,y\right) = \sin \left(\pi x \right) \sin \left( \pi y \right)$ to the steady scalar transport problem.   This solution is realized on the domain $\Omega = \left[0,1\right] \times \left[0,1\right]$ by setting the velocity to $\bm{a} = \left( \sqrt{2}/2, \sqrt{2}/2 \right)$, setting the forcing to:
\begin{equation}
f \left(x,y\right) = \frac{\pi \sqrt{2}}{2} \left( \cos \left( \pi x \right) \sin \left( \pi y \right) + \sin \left( \pi x \right) \cos \left( \pi y \right) \right) + \pi^2 \kappa \left( 2 \sin \left( \pi x \right) \sin \left( \pi y \right) \right),
\end{equation}
and applying homogeneous Dirichlet boundary conditions along the entire domain boundary.  To ensure the problem is advection-dominated, the diffusion is set to $\kappa = 10^{-6}$.  The problem is then characterized by the P\'{e}clet number $Pe = \frac{| \bm{a} | L}{\kappa} = 10^{6}$ where $L$ is the length of the domain.  Numerical solutions are calculated on grids of $2 \times 2$, $4 \times 4$, $8 \times 8$, $16 \times 16$, $32 \times 32$, $64 \times 64$, and $128 \times 128$ elements for coarse-scale polynomial degrees $p = 1, 2$ and subscale polynomial degrees $p_f = 1, 2$.  The $L^2$-norms of the error in the coarse-scale solution and the subscale solution are presented in Fig. \ref{fig:convergence}.  We observe that both the coarse-scale and subscale solutions converge optimally under mesh refinement.

\subsection{Steady advection skew to the mesh}

We next consider the classical two-dimensional advection skew to mesh problem.  This problem is graphically depicted in Fig. \ref{fig:skew_setup}.  The P\'{e}clet number for the problem is $Pe = 10^6$, so the problem is advection-dominated.  Throughout this subsection, we select $\theta = 45^{\circ}$.

\begin{figure}[b!]
\begin{center}
\includegraphics[width=10cm]{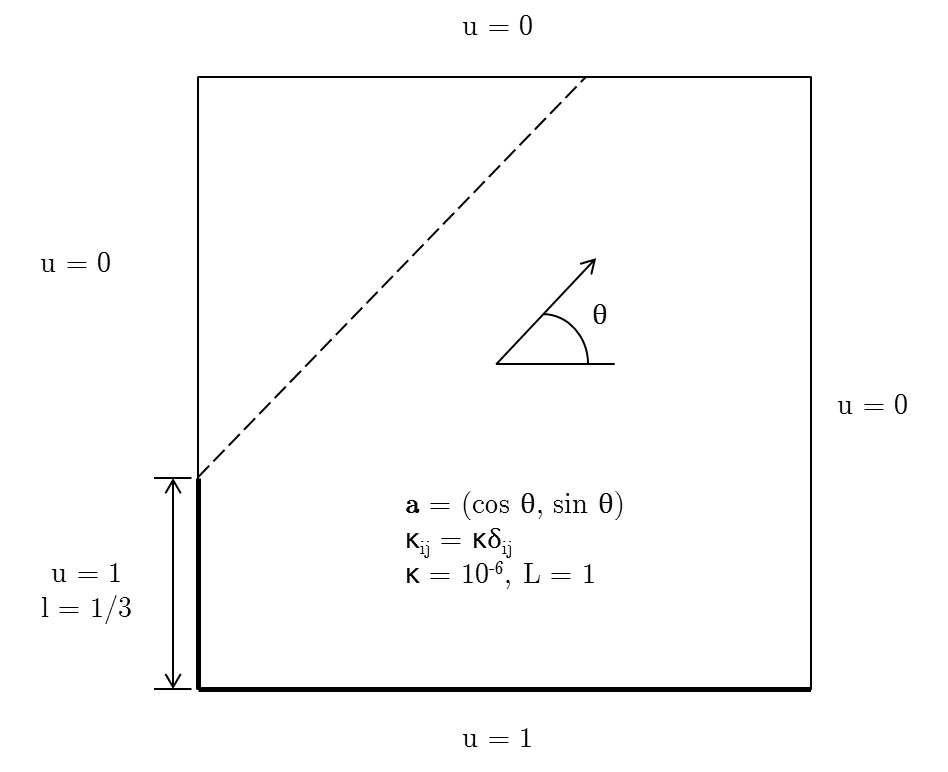}
\caption{Problem setup for the two-dimensional advection skew to the mesh problem.}
\label{fig:skew_setup}
\end{center}
\end{figure}

First, we explore the effect of mesh refinement.  Solutions on uniform NURBS meshes with $32 \times 32$, $64 \times 64$, and $128 \times 128$ elements for both $p = 1$ and $p = 2$ are computed and compared.  For each of the cases, the subscale polynomial degree is chosen to be $p_f = 1$.
A representative solution for $p = 2$, $p_f = 1$, and a $64 \times 64$ mesh is shown in Fig. \ref{fig:p2_pf1_skew}.  The effect of mesh refinement is illustrated by comparison of the obtained solutions along the slice $y = 0.7$ in Fig. \ref{fig:skew_slice}.  The results are as expected.  Refinement of the mesh allows for a more accurate representation of the boundary and internal layers.  The results are polluted by the presence of oscillations near the boundary and internal layers, but for $p = 1$, these oscillations are limited to a region of one or two elements away from the layers.  For $p = 2$, the oscillations do infiltrate further into the domain, but it is our experience that these oscillations are not due to method instability but rather the fact that we are trying to fit the sharp gradients present in the boundary and internal layers.  As isogeometric NURBS discretizations exhibit a high level of continuity, any oscillations due to overfitting such gradients are expected to extend into the domain in analogy with Gibbs' phenomena.  Nonetheless, the oscillations decay in magnitude away from the layers and are still limited to a region of a fixed number of elements away from the layers.  As we will later see, these oscillations can be greatly suppressed through the use of weak enforcement of Dirichlet boundary conditions.

\begin{figure}[t]
  \begin{center}
      \subfigure[Isometric view]{\includegraphics[width=8cm]{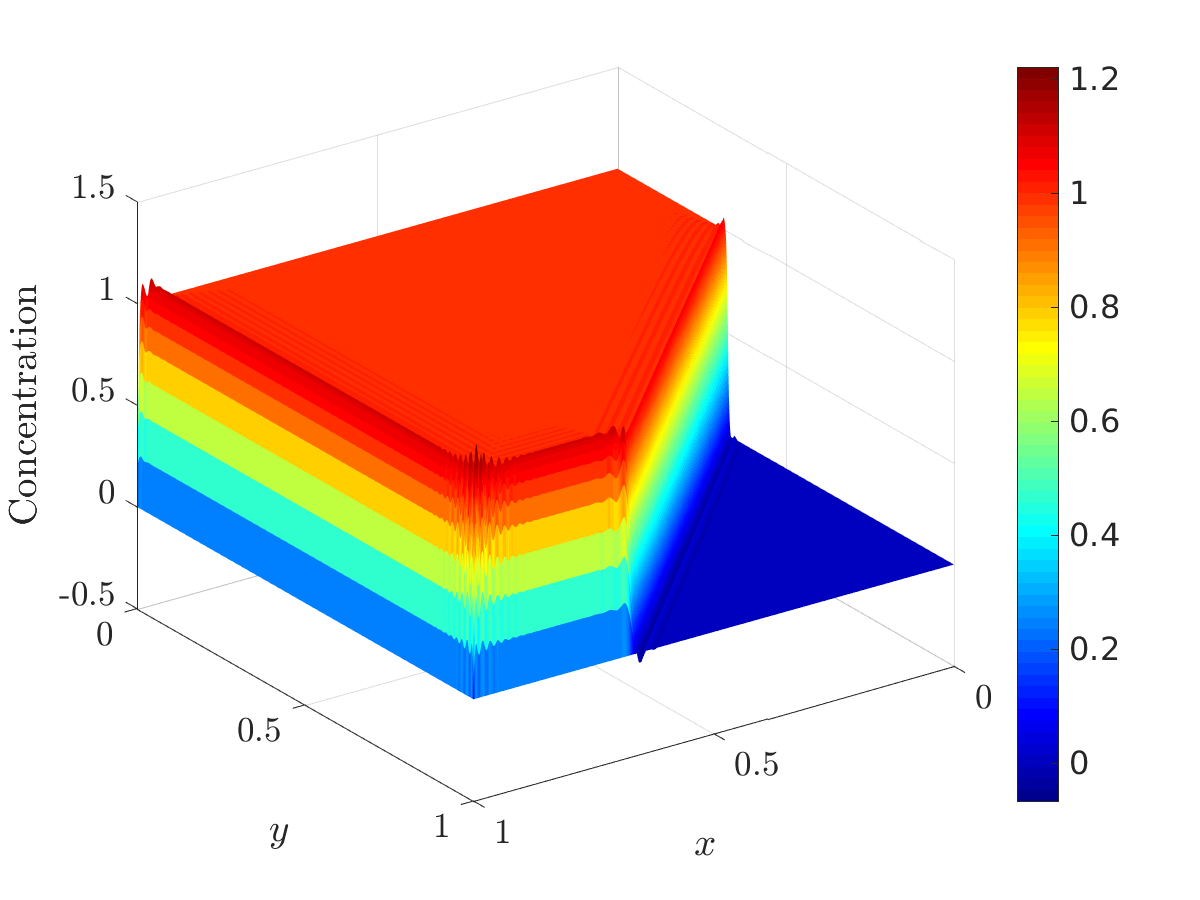}}
      \subfigure[Overhead view]{\includegraphics[width=8cm]{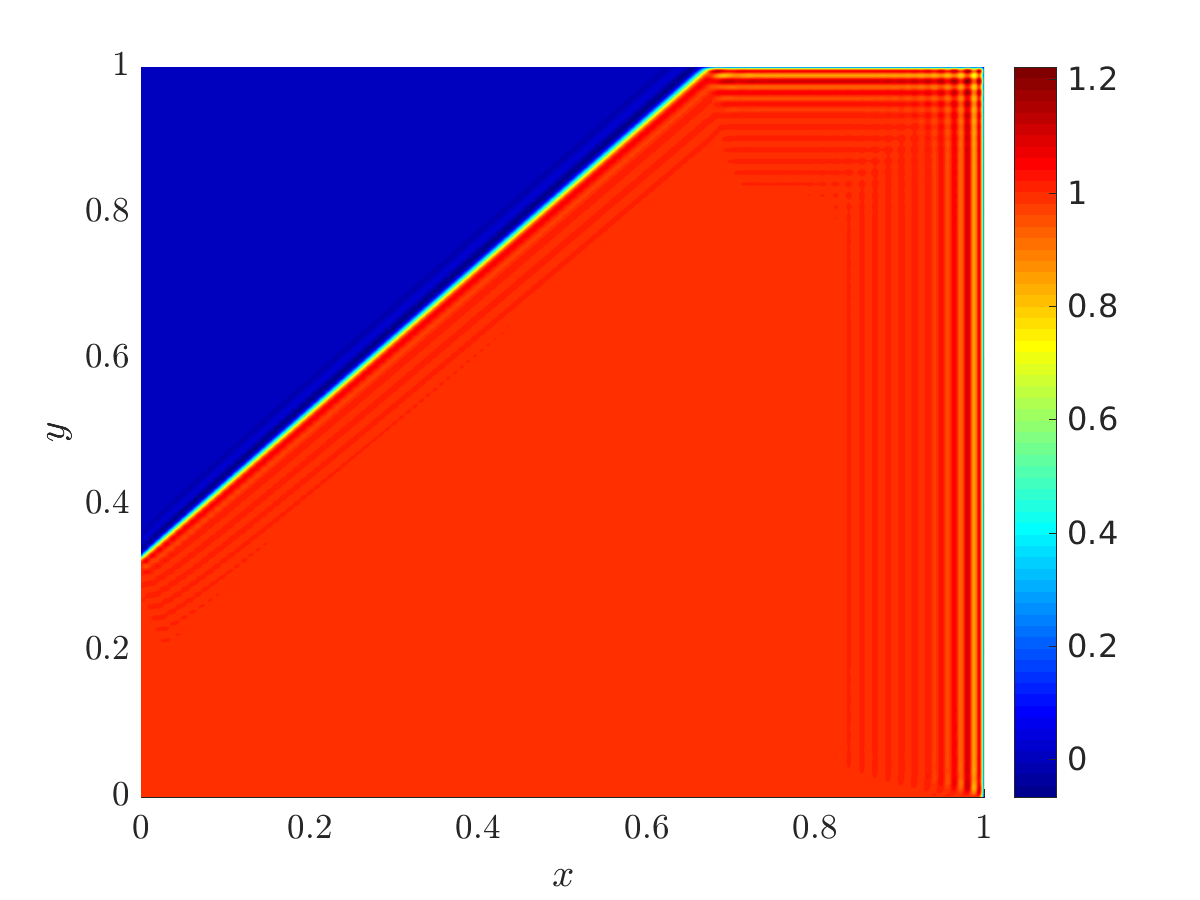}}
      \subfigure[Isometric view]{\includegraphics[width=8cm]{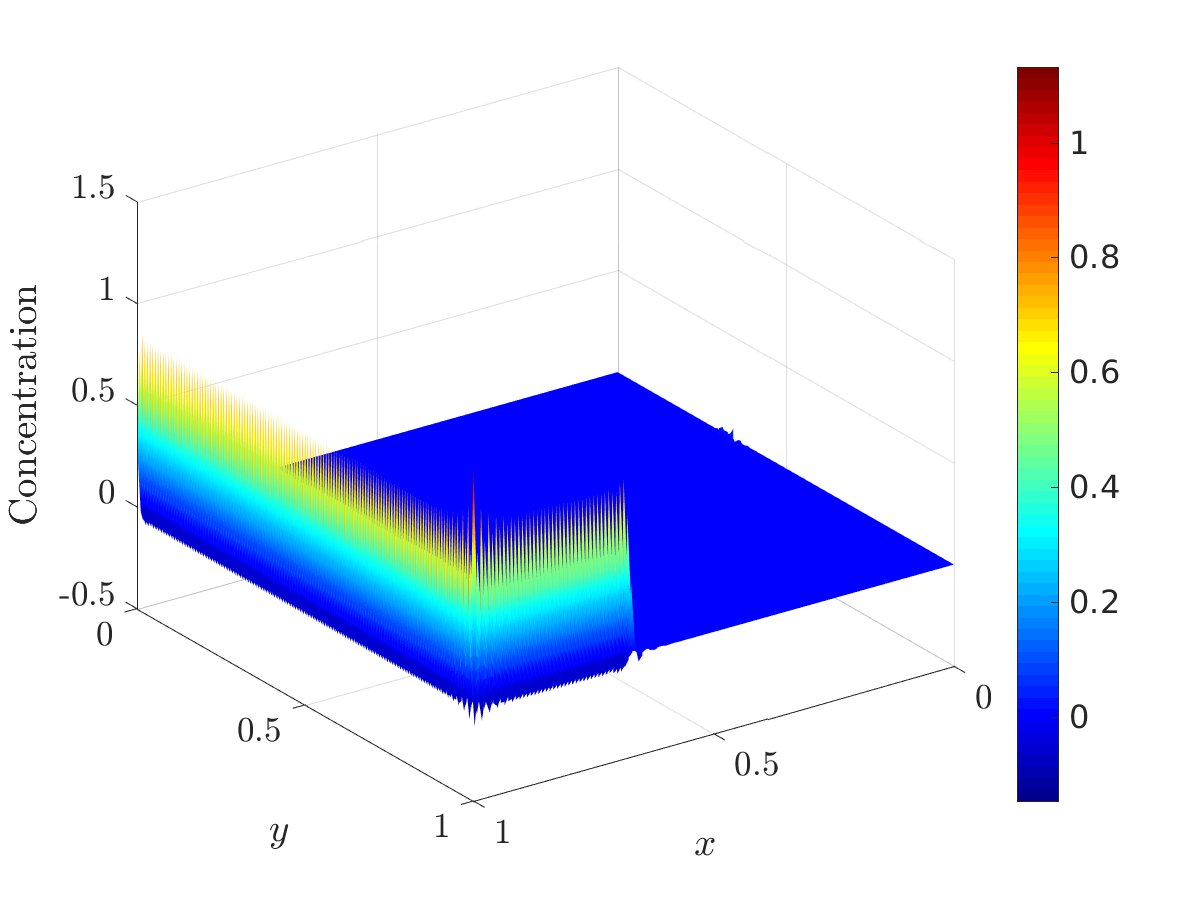}}
      \subfigure[Overhead view]{\includegraphics[width=8cm]{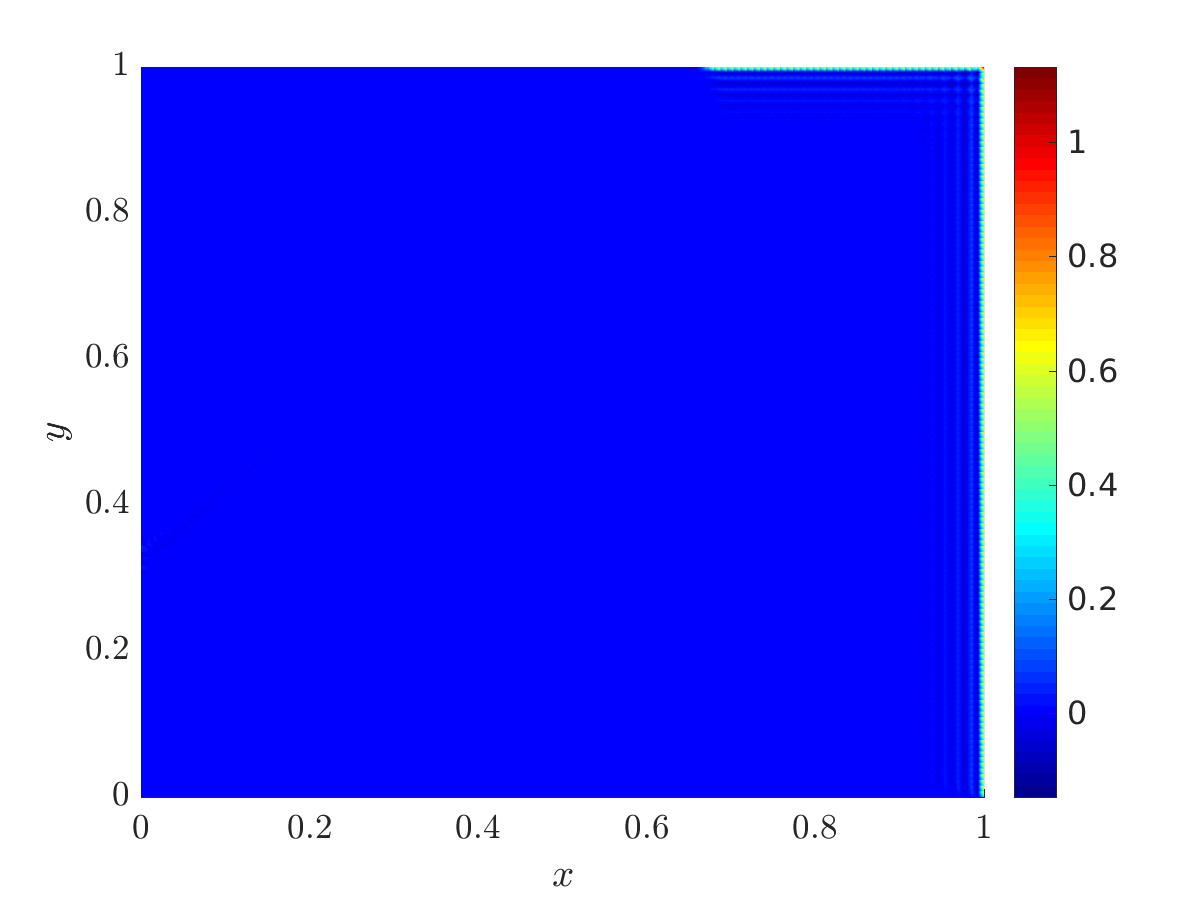}}
    \caption{Coarse-scale (top) and subscale (bottom) solutions for the advection skew to the mesh problem for $p = 2$, $p_f = 1$, and $h = 1/64$.}
    \label{fig:p2_pf1_skew}
  \end{center}
\end{figure}

\begin{figure}[ht]
\begin{center}
      \subfigure[Coarse-scale solution for $p = 1$, $p_f = 1$, and varying mesh size.]{\includegraphics[width=8cm]{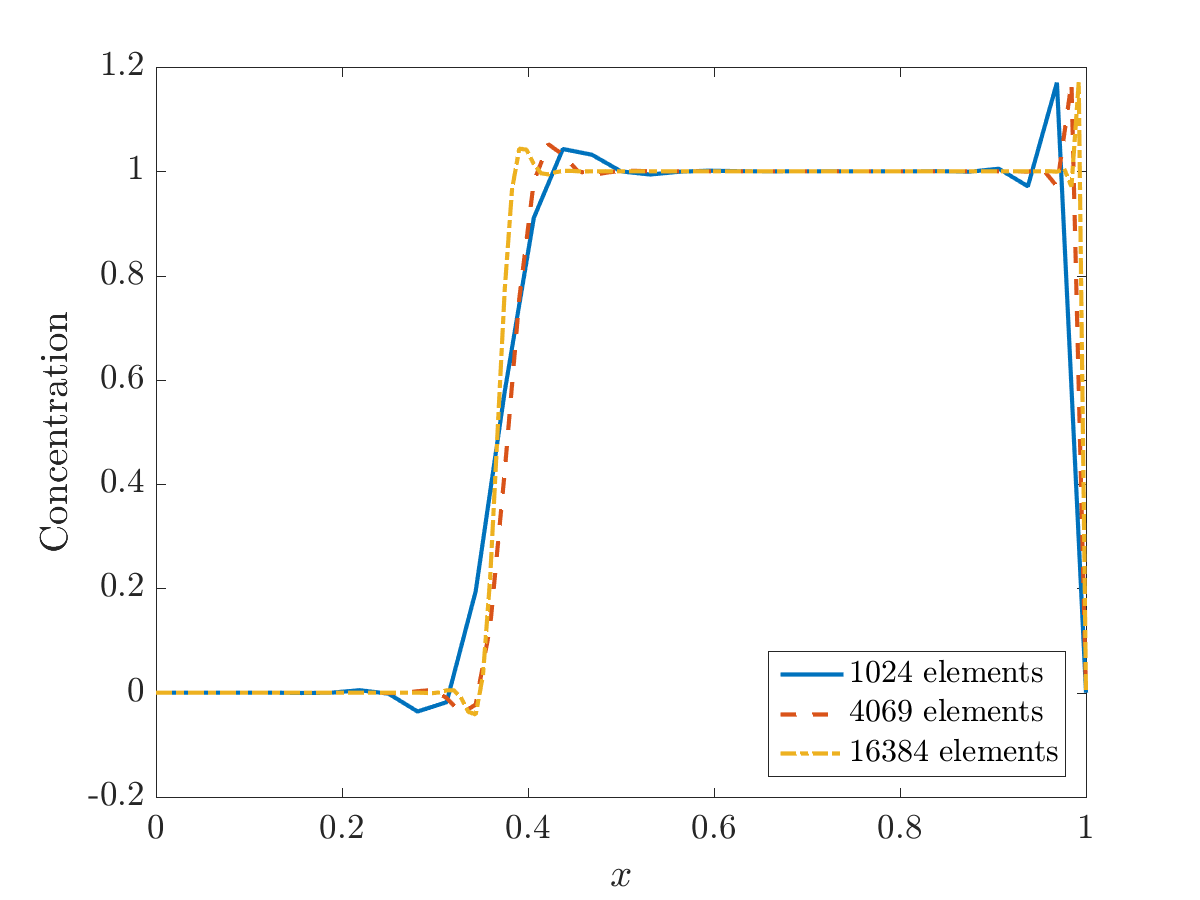}}
      \subfigure[Coarse-scale solution for $p = 2$, $p_f = 1$, and varying mesh size.]{\includegraphics[width=8cm]{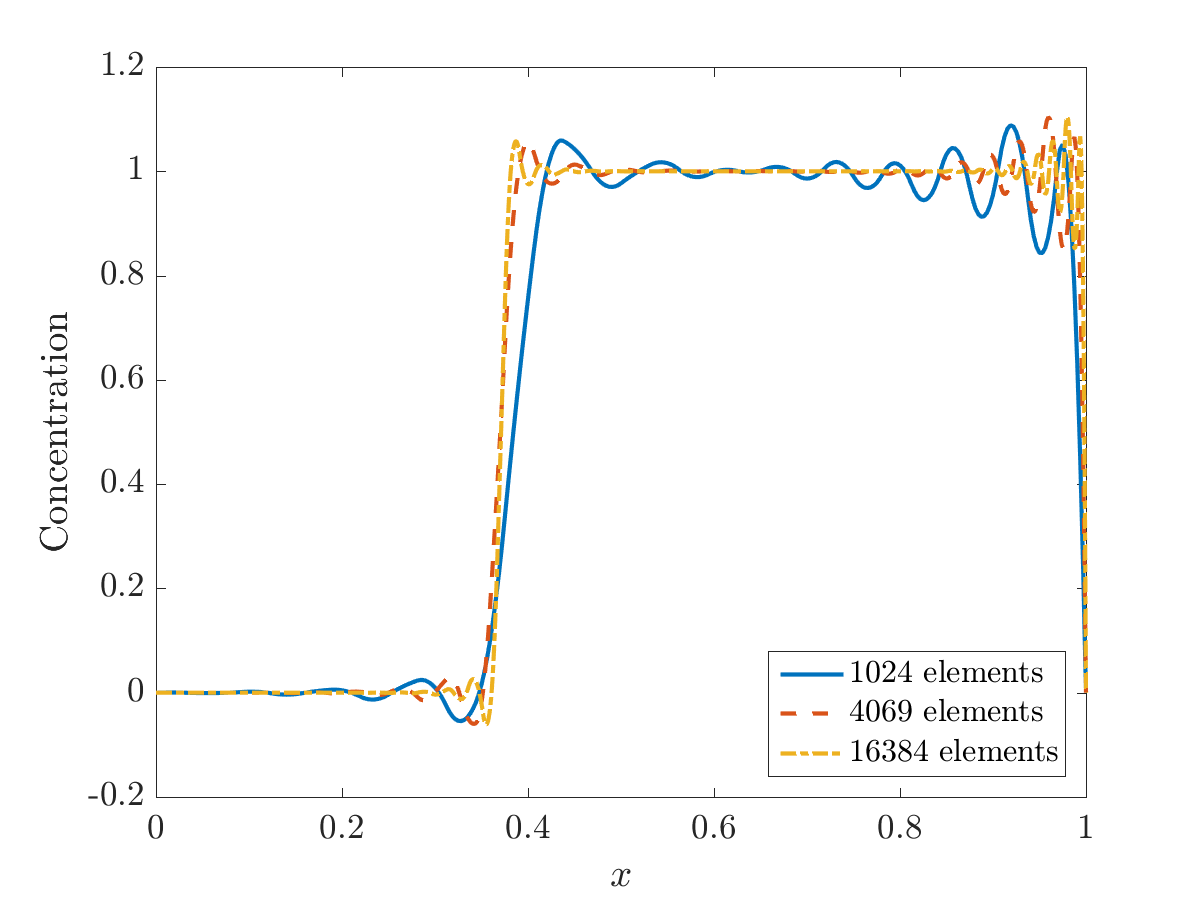}}
\caption{Coarse-scale solution along $y = 0.7$ for varying mesh sizes and coarse-scale polynomial degrees.}
\label{fig:skew_slice}
\end{center}
\end{figure}

We now explore the effect of degree elevation.  First, we examine the effect of coarse-scale degree elevation.  Solutions for $p = 1, 2, 3$ with $p_f = 1$ on a uniform NURBS mesh with $64 \times 64$ elements are computed and compared.  The effect of coarse-scale degree elevation is illustrated by comparison of the obtained solutions along the slice $y = 0.7$ in Fig. \ref{fig:skew_slice_degree_elevation}.  Note that each of the coarse-scale solutions are able to capture the internal and boundary layers, and the coarse-scale solutions associated with higher polynomial degrees exhibit sharper representations of the layers.  However, the coarse-scale solutions associated with higher polynomial degrees also exhibit oscillations in the regions near the internal and boundary layers due to overfitting and Gibbs' phenomena.  Next, we explore the effect of subscale degree elevation.  Solutions for $p = 2$ with $p_f = 1, 2, 3$  are computed on a uniform NURBS mesh with $64 \times 64$ elements and compared.  The effect of fine-scale degree elevation is illustrated by comparison of the obtained solutions along the slice $y = 0.7$ in Fig. \ref{fig:skew_slice_degree_elevation}.  From the figure, it is seen that the choice of subscale polynomial degree does not greatly affect the coarse-scale solution.  Consequently, it is recommended that one choose the lowest possible subscale polynomial degree in order to stabilize the higher-order coarse-scale solution at minimal computational cost.

\begin{figure}[ht]
\begin{center}
      \subfigure[Coarse-scale solution for $p_f = 1$, $h = 1/64$, and varying coarse-scale polynomial degrees.]{\includegraphics[width=8cm]{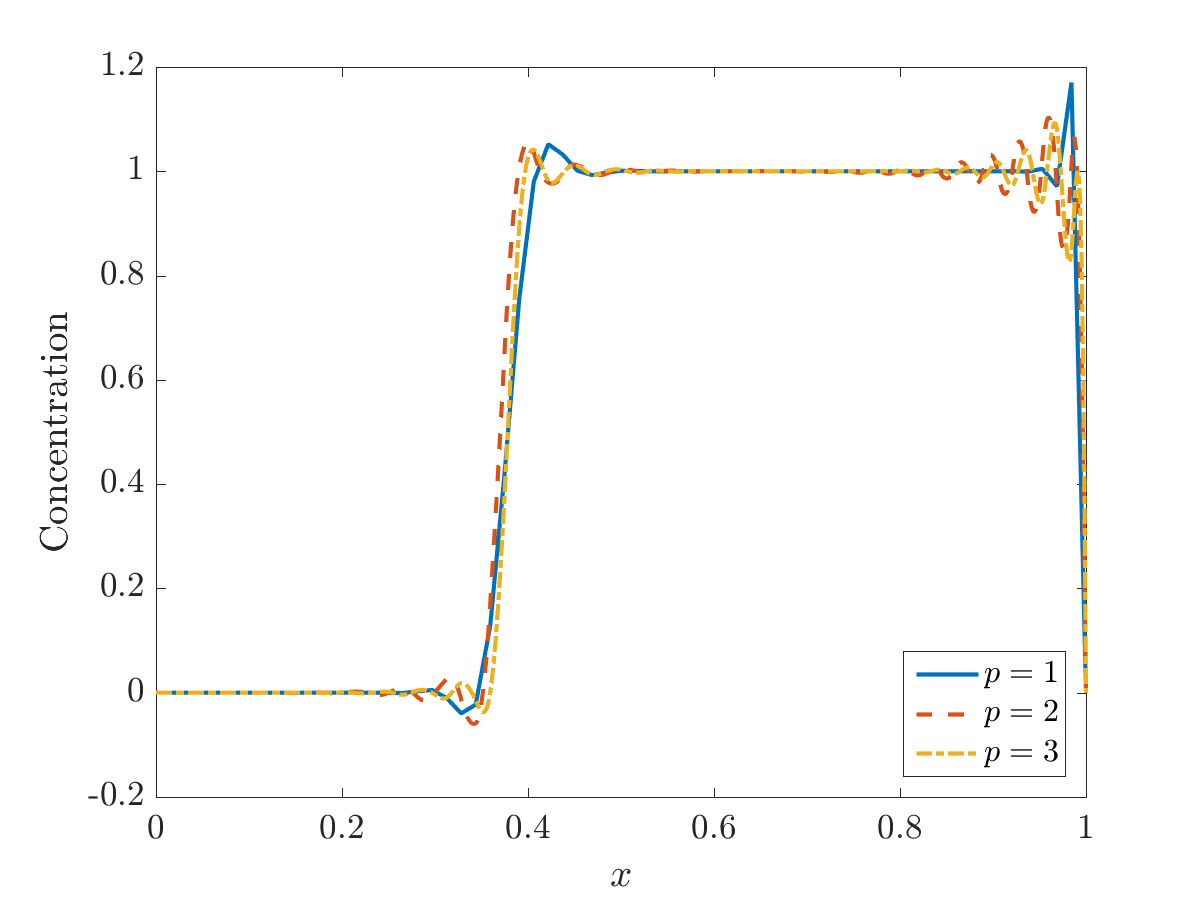}}
      \subfigure[Coarse-scale solution for $p = 2$, $h = 1/64$, and varying subscale polynomial degrees.]{\includegraphics[width=8cm]{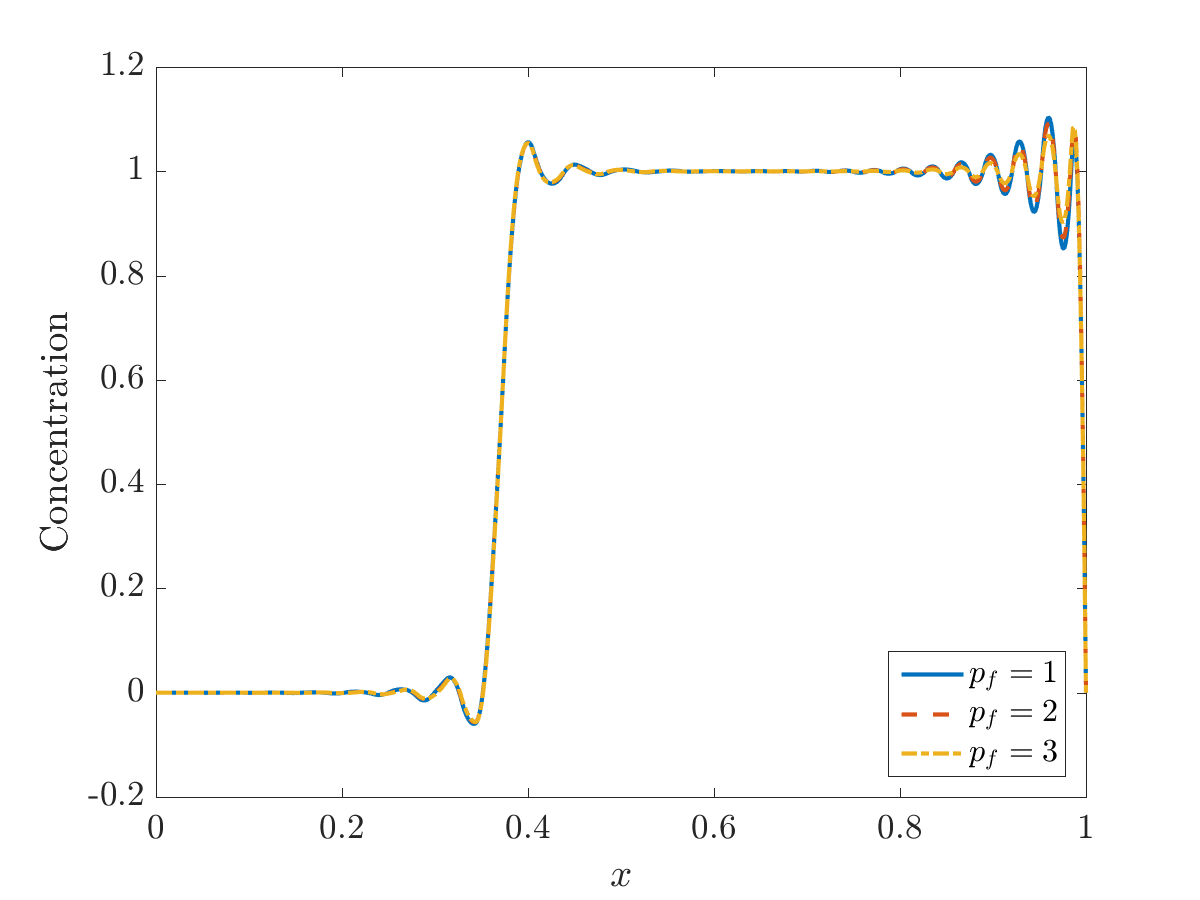}}
\caption{Coarse-scale solution along $y = 0.7$ for varying coarse-scale and subscale polynomial degrees.}
\label{fig:skew_slice_degree_elevation}
\end{center}
\end{figure}

Heretofore, we have enforced the Dirichlet boundary conditions on the coarse-scale solution field in a strong manner.  However, it has been noted in recent work that greatly enhanced results may be obtained by instead imposing these in a weak manner using a combination of upwinding and Nitsche's method \cite{Bazilevs07a,Bazilevs07c}.  The resulting coarse-scale governing equation takes the form:\\

\noindent \textbf{Coarse-Scale Governing Equation with Weak Boundary Condition Treatment:}
\begin{equation}
\begin{aligned}
\int_\Omega \frac{\partial \bar{u}}{\partial t} \bar{v} -  \int_\Omega \left(\bar{u} + u'\right) \bm{a} \cdot \nabla \bar{v} +  \int_\Omega \kappa \nabla \bar{u} \cdot \nabla \bar{v} + \int_{\Gamma_N} a_n \bar{u} \bar{v} \\ 
- \int_{\Gamma_D} \kappa \nabla \bar{u} \cdot \bm{n} \bar{v} - \int_{\Gamma_D} \bar{u} \kappa \nabla \bar{v} \cdot \bm{n} + \sum_{K = 1}^{n_{el}} \int_{\Gamma_K \cap \Gamma_D} \frac{C_{Nitsche} \kappa}{h_K} \bar{u} \bar{v} + \int_{\Gamma_D} a^+_n \bar{u} \bar{v} \\
= \int_\Omega f \bar{v} + \int_{\Gamma_N} h \bar{v} - \int_{\Gamma_D} g \kappa \nabla \bar{v} \cdot \bm{n} + \sum_{K = 1}^{n_{el}} \int_{\Gamma_K \cap \Gamma_D} \frac{C_{Nitsche} \kappa}{h_K} g \bar{v} - \int_{\Gamma_D} a^-_n g \bar{v} \\
\textup{ for all } \bar{v} \in \bar{\mathcal{V}}
\end{aligned}
\end{equation}

\noindent In the above equation, $C_{Nitsche}$ is a penalty constant that must be chosen sufficiently large for method stability, $a^+_n = \left(\bm{a} \cdot \bm{n}\right)^+ = \max\left\{ \bm{a} \cdot \bm{n}, 0 \right\}$, and $a^-_n = \left(\bm{a} \cdot \bm{n}\right)^- = \min\left\{ \bm{a} \cdot \bm{n}, 0 \right\}$.  Herein, we choose $C_{Nitsche} = 4p$.  It should be mentioned that the discontinuous subscale model remains untouched if one elects to weakly enforce Dirichlet boundary conditions on the coarse-scale solution.

To assess the effect of weak boundary condition enforcement, we have computed solutions using both using strongly-enforced and weakly-enforced Dirichlet boundary conditions with $p = 2$ and $p_f = 1$  on a uniform NURBS mesh with $64 \times 64$ elements.  The effect of boundary condition enforcement is illustrated by comparison of the obtained solutions along the slice $y = 0.7$ in Fig. \ref{fig:skew_slice_BC_enforcement}.  Note from the figure that there no longer remain any oscillations near the boundary layer if one employs a weak boundary condition enforcement.  Moreover, while there remain oscillations near the internal layer, these oscillations are small in magnitude and limited to a region of a small number of elements away from the layer.  Consequently, it is strongly advised that one weakly enforces Dirichlet boundary conditions along portions of the boundary where a layer is expected.

\begin{figure}[t]
\begin{center}
\includegraphics[width=10cm]{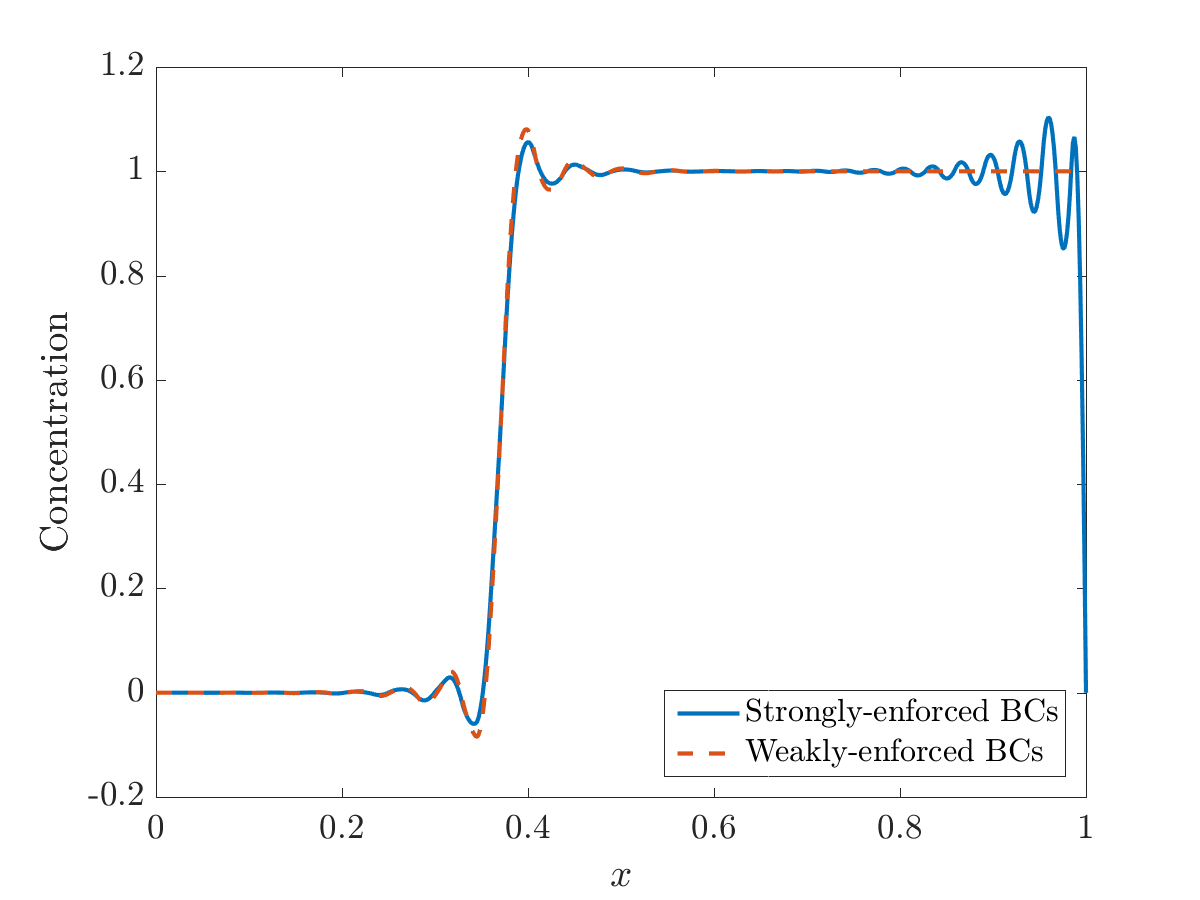}
\caption{Coarse-scale solution along $y = 0.7$ for varying boundary-condition enforcement with $p = 2$, $p_f = 1$, and $h = 1/64$.}
\label{fig:skew_slice_BC_enforcement}
\end{center}
\end{figure}

\subsection{Steady advection in a quarter-annulus}

We next consider a problem posed on a non-square geometry, namely the advection of a Gaussian curve in a quarter annulus.  This problem is graphically depicted in Fig. \ref{fig:annulus_setup}.  Here, $r_i$ is the inner radius of the annulus and $r_o$ is the outer radius of the annulus and they are chosen to be $r_i = 1$ and $r_o = 2$.  The flow field for this problem is chosen to be $\bm{a} = \left( y, -x \right)$ and the diffusivity is chosen as $\kappa = 1 \times 10^6$ such that the problem is advection-dominated.  Dirichlet boundary conditions are strongly enforced on the lower boundary of the annulus with:
\begin{equation}\label{eq:annulus_bc}
g\left(x,y=0\right) = \frac{1}{\sigma \sqrt{2 \pi}} \exp \left(\frac{(-x-\mu)^2}{2 \sigma^2}\right),
\end{equation}
where $\sigma = 0.05$ and $\mu = -1.5$.  Traction-free Neumann boundary conditions are prescribed on the other three boundaries.  A typical solution corresponding to $p = 2$, $p_f = 1$, and a NURBS mesh of $64 \times 64$ elements is displayed in Fig. \ref{fig:annulus_isometric}.  From the figure, it is clear that the method is able to very accurately capture the solution free of oscillations on the curved geometry.

\begin{figure}[t]
\begin{center}
\includegraphics[width=8cm]{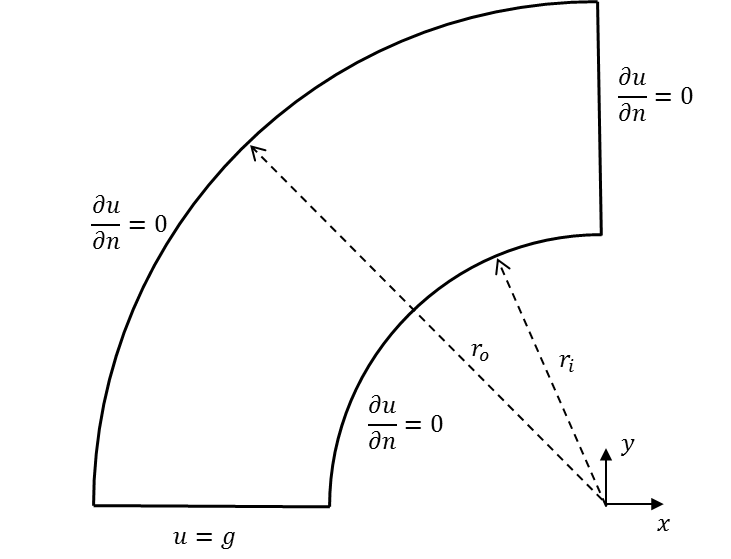}
\caption{Problem setup for the advection in a quarter-annulus problem.}
\label{fig:annulus_setup}
\end{center}
\end{figure}

\begin{figure}[t]
  \begin{center}
      \subfigure[Isometric view]{\includegraphics[width=8cm]{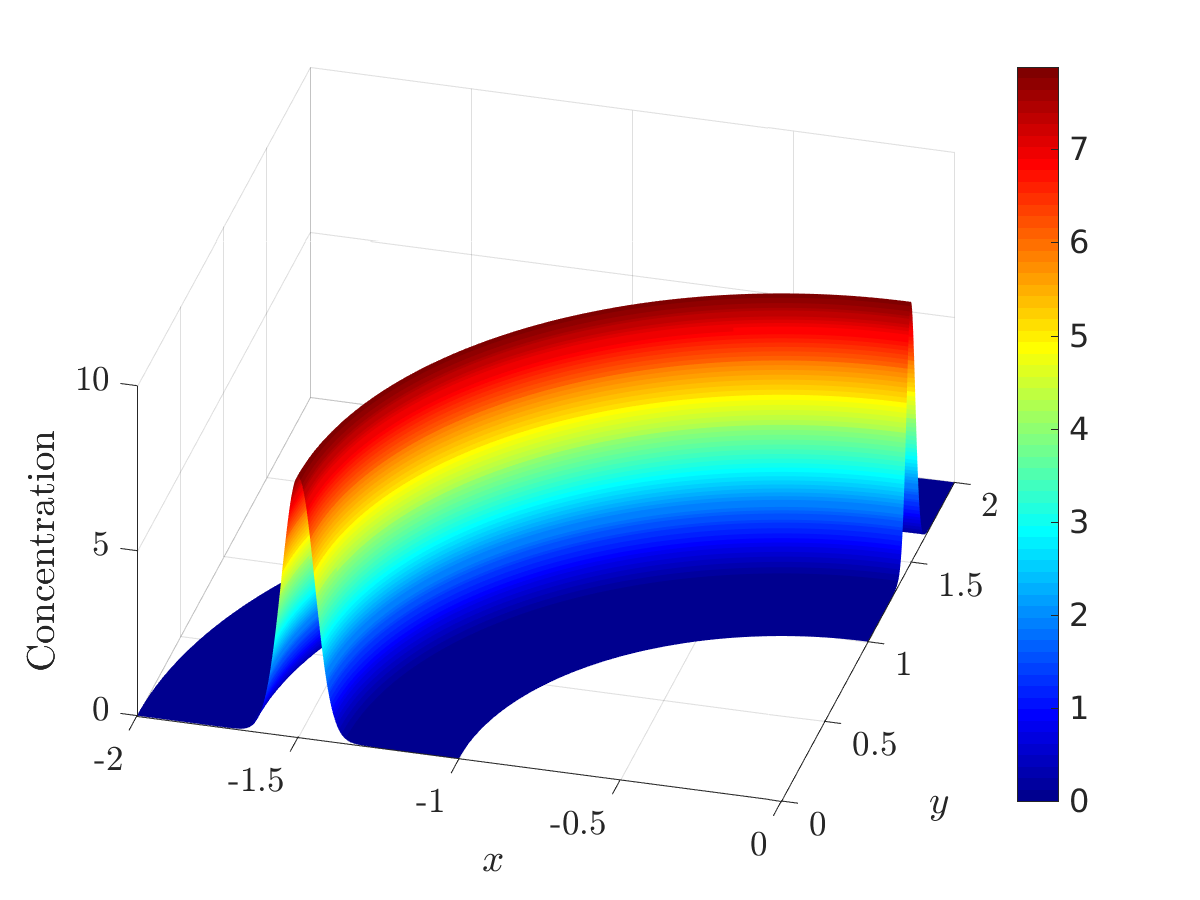}}
      \subfigure[Overhead view]{\includegraphics[width=8cm]{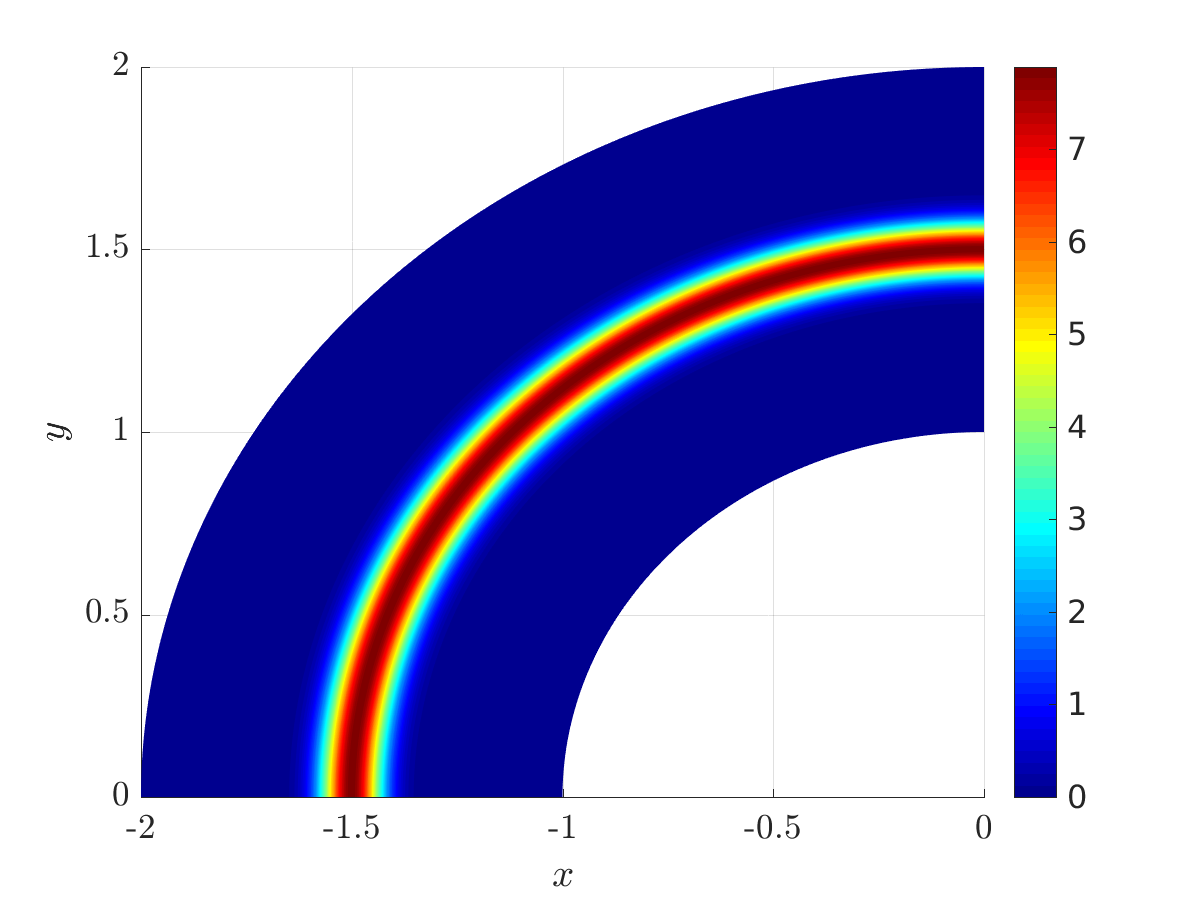}}
      \subfigure[Isometric view]{\includegraphics[width=8cm]{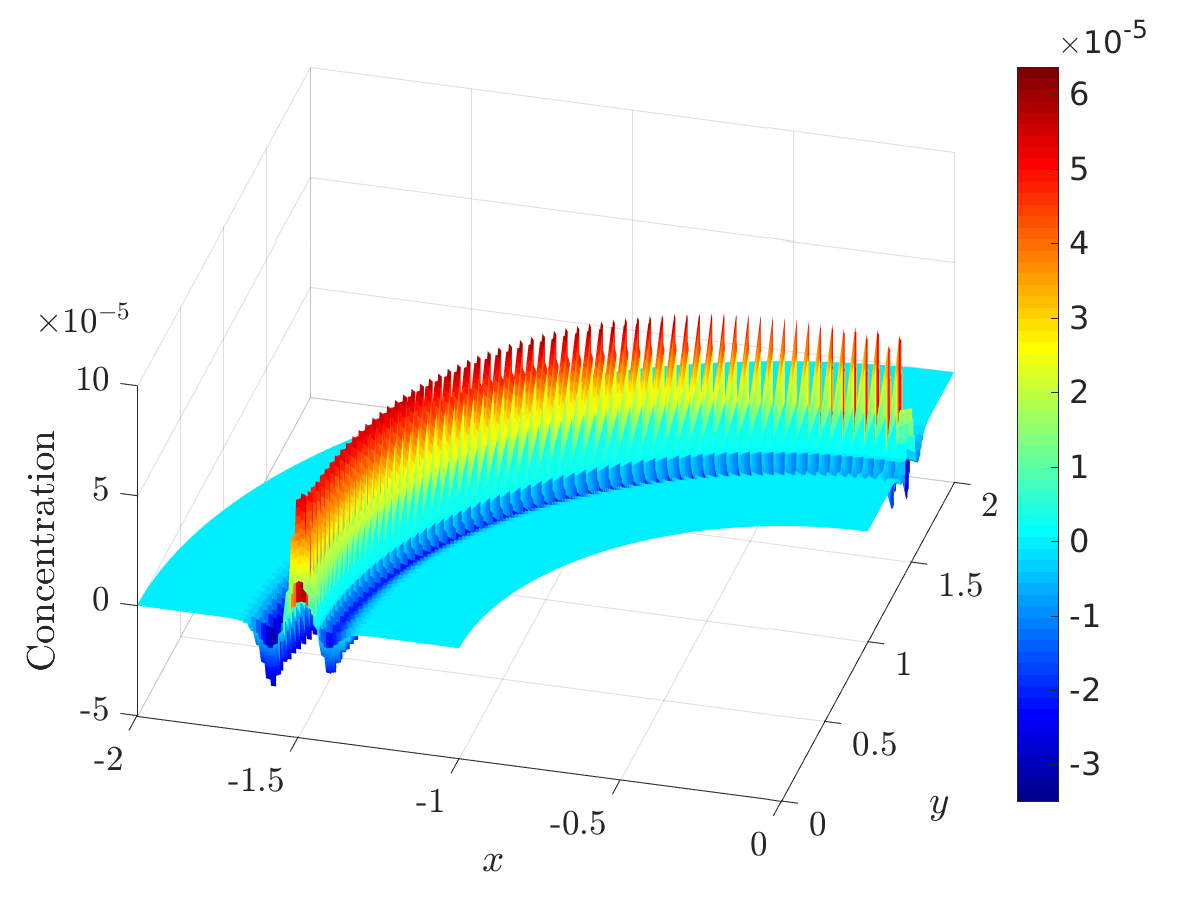}}
      \subfigure[Overhead view]{\includegraphics[width=8cm]{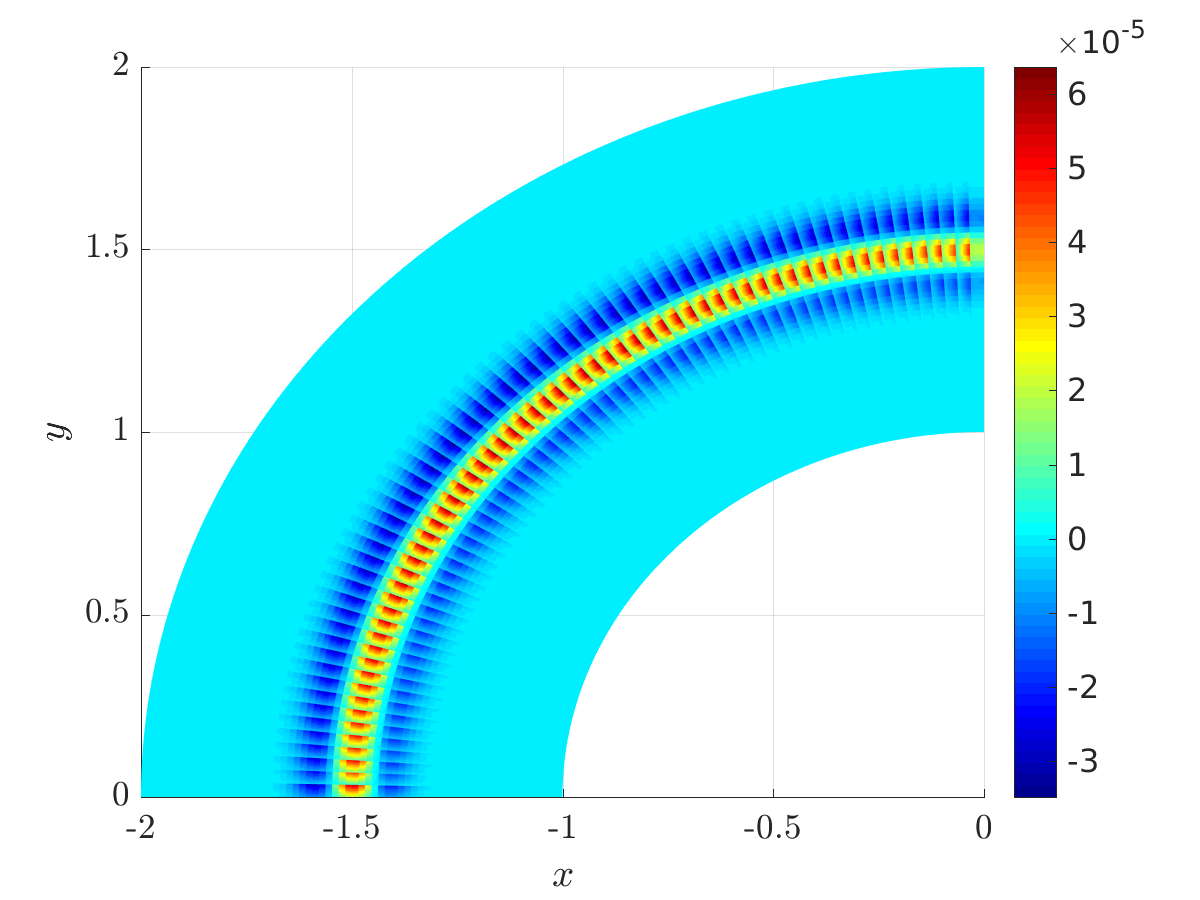}}
    \caption{Coarse-scale (top) and subscale (bottom) solutions for the advection in a quarter-annulus problem for $p = 2$, $p_f = 1$, and a NURBS mesh of $64 \times 64$ elements.}
    \label{fig:annulus_isometric}
  \end{center}
\end{figure}

\subsection{Unsteady advection of a Gaussian hill}

We now consider an unsteady scalar transport problem, namely the advection of a Gaussian hill in a rotating flow field.  For this problem, the domain is $\Omega = [0, 1] \times [0, 1]$ and the flow field is given as $\bm{a} = \left( y- 0.5, x - 0.5 \right)$.  The initial scalar field is given as:
\begin{equation}\label{eq:gaussian_ic}
u\left(x,y,0\right) = \exp\left(-\frac{1}{2 \sigma} \left( \left(x - 0.5 \right)^2 + \left(y - 0.25 \right)^2 \right) \right),
\end{equation}
where $\sigma = 0.065$, and Dirichlet boundary conditions are strongly enforced along the entire boundary of the domain.  Numerical solutions are calculated on a $64 \times 64$ NURBS mesh with coarse-scale polynomial degree $p = 2$ and subscale polynomial degree $p_f = 1$.  Both the dynamic and quasi-static discontinuous subscale models were employed to stabilize the solution.  Time-integration is carried out with the generalized-alpha method using a time step of $\Delta t = 0.005$ and $\rho_{\infty} = 1$ such that $\alpha_m = 0.5$, $\alpha_f = 0.5$, and $\gamma = 0.5$.  The small time step is chosen such that the only source of error is spatial error.

Establishing the initial condition for this problem takes some care, and the approach differs between the dynamic and quasi-static models.  For the quasi-static model, we do not require a time-history of the discontinuous subscale solution.  Thus, only an initial condition for the coarse-scale solution field is required, and this is established through an $L^2$-projection of the exact initial condition.  For the dynamic model, a time-history of the discontinuous subscale solution is required.  Consequently, after an initial condition for the coarse-scale solution field is obtained, an initial condition for the discontinuous subscale solution is established through an $L^2$-projection of the difference between the exact initial condition and the coarse-scale initial condition.

\subsubsection{Results for the Quasi-static Discontinuous Subscale Model}

We first display results obtained using the quasi-static discontinuous subscale model.  In Fig. \ref{fig:gaussian_centerline}, the concentration along the centerline of the hill (found at $y = 0.25$) is plotted for the initial solution and the solution after 1, 5, and 10 full revolutions ($t = 6.28, 31.40, 62.80$, respectively).  It is clear that the method is able to preserve the hill with very high accuracy even after several rotations.  After one rotation, the peak is reduced to $0.997$, representing a numerical dispersion of the peak of 0.3\%.  After ten rotations, the peak is reduced to $0.969$, representing a numerical dispersion of the peak of 3.1\%.  There is also some slight skewing of the solution in the downwind direction.

\subsubsection{Results for the Dynamic Discontinuous Subscale Model}

We next display results obtained using the dynamic discontinuous subscale model.
In Fig. \ref{fig:gaussian_centerline}, the concentration along the centerline of the hill (found at $y = 0.25$) is plotted for the initial solution and the solution after 1, 5, and 10 full revolutions ($t = 6.28, 31.40, 62.80$, respectively).  Like the quasi-static model, it is clear that the method is superb in preserving the hill.  After one rotation, the peak is reduced to $0.996$, representing a numerical dispersion of the peak of 0.4\%.  After ten rotations, the peak is reduced to $0.965$, representing a numerical dispersion of the peak of 3.5\%.  Again, like the quasi-static model, there is some slight skewing of the solution in the downwind direction.  Curiously, the magnitude of the subscale solution is three orders less for the dynamic model than the quasi-static model, but the two models produce very similar results both quantitatively and qualitatively.

\begin{figure}[t]
  \begin{center}
      \subfigure[Coarse-scale solution for the quasi-static model]{\includegraphics[width=8cm]{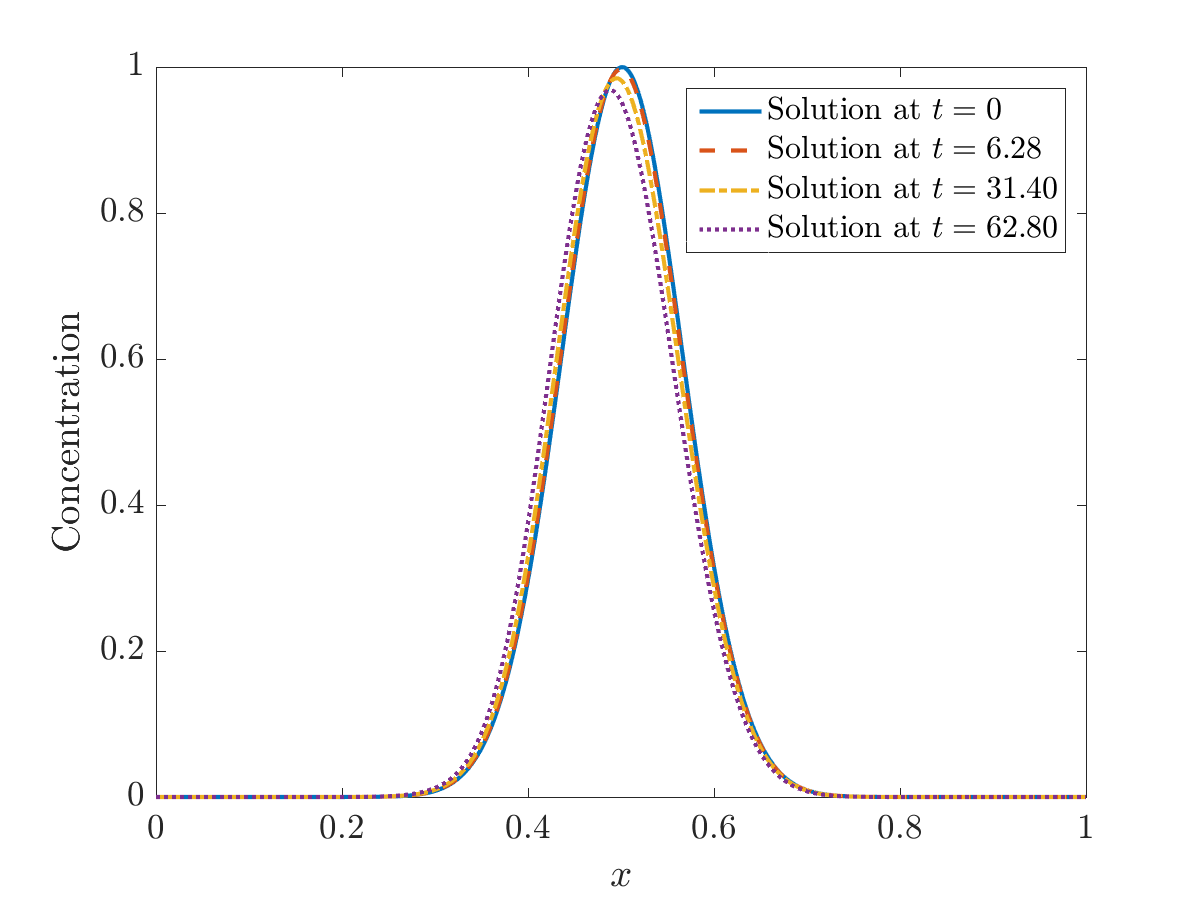}}
      \subfigure[Coarse-scale solution for the dynamic model]{\includegraphics[width=8cm]{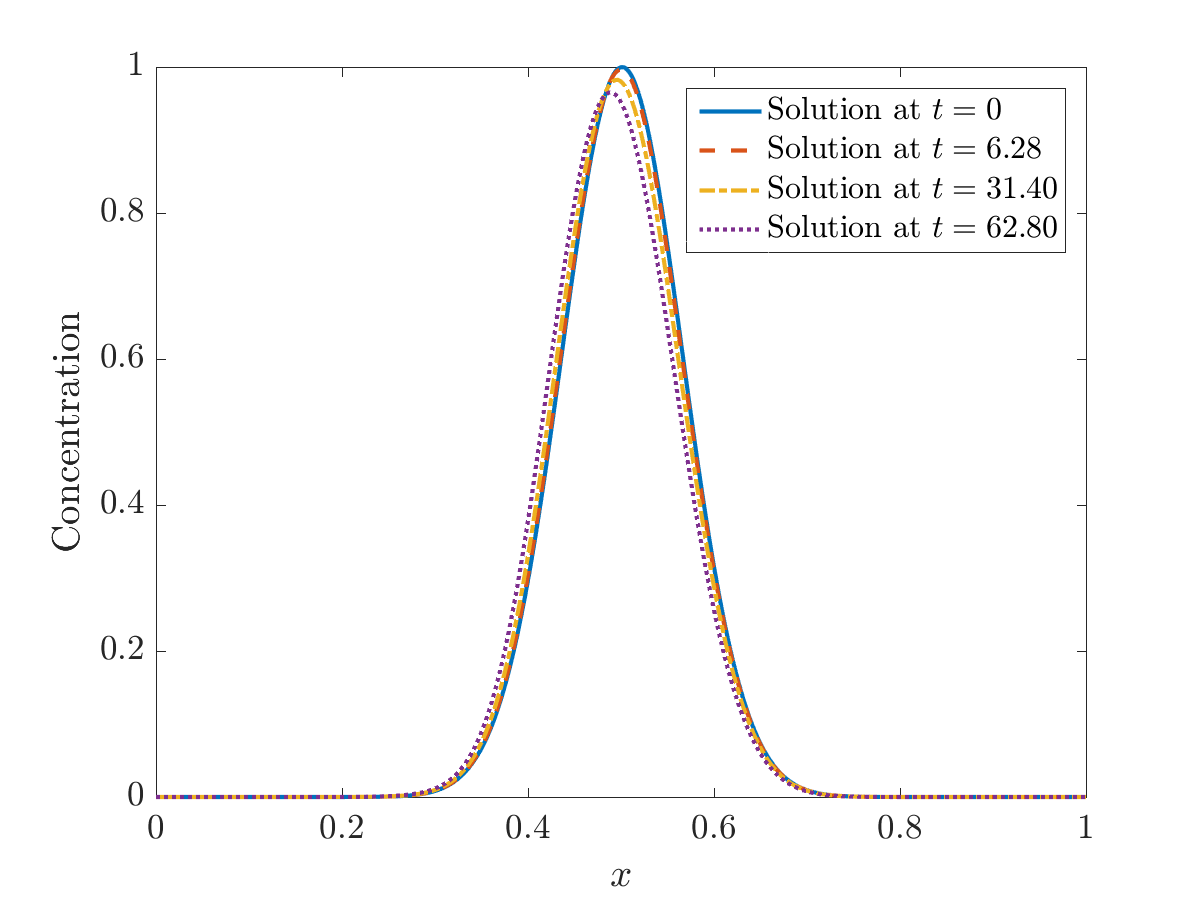}}
	\caption{Coarse-scale solution for the unsteady advection of a Gaussian hill problem along $y = 0.25$ for $p = 2$, $p_f = 1$, and $h = 1/64$.}
	\label{fig:gaussian_centerline}
  \end{center}
\end{figure}

\subsection{Unsteady advancing front}

We finish by returning back to the advection skew to the mesh problem but in an unsteady context.  Namely, we initialize the scalar field to be zero everywhere except at the Dirichlet boundary where the initial condition meets the specified boundary condition.  We then let the solution evolve according to the problem specifications in Fig. \ref{fig:skew_setup}.  Since we expect a boundary layer to form for this problem, we elect to enforce the Dirichlet conditions in a weak fashion.  For all our numerical results, a uniform NURBS mesh with $64 \times 64$ square elements was employed with 
with a coarse-scale polynomial degree of $p = 2$ and a subscale polynomial degree of $p_f = 1$.  Time-integration is carried out using the generalized-alpha method using a time step of $\Delta t = 0.01$ and $\rho_{\infty} = 1$ such that $\alpha_m = 0.5$, $\alpha_f = 0.5$, and $\gamma = 0.5$.  We compute and compare solutions using both the dynamic and quasi-static discontinuous subscale models.

\subsubsection{Results for the Quasi-static Discontinuous Subscale Model}

Solutions for the advancing front problem solved with the quasi-static discontinuous subscale model at representative time steps of $t = 0, 0.51, 2.01$ are shown in Figs. \ref{fig:advancingfront_coarse} and \ref{fig:advancingfront_fine}.  The coarse-scale solution converges to the steady-state solution as expected.  The coarse-scale solution is quite accurate in the eyeball norm, though at each time step, minor oscillations are present near the internal layer similar to what was observed in the steady skew to the mesh problem.  Additionally, the magnitude of the oscillations in the fine-scale solution at the final time-step correspond to roughly the same magnitude of the oscillations observed in the steady problem.  This is expected, since we expect the steady state subscale solution to coincide with the subscale solution in the steady problem.

\subsubsection{Results for the Dynamic Discontinuous Subscale Model}

Solutions for the advancing front problem solved with the dynamic discontinuous subscale model at representative time steps of $t = 0, 0.51, 2.01$ are shown in Figs. \ref{fig:advancingfront_coarse} and \ref{fig:advancingfront_fine}.  Again, the coarse-scale solution is quite accurate and converges to the steady-state solution as expected, and minor oscillations are present near the internal layer similar to what was observed in the steady skew to the mesh problem.  Surprisingly, the magnitude of the subscale solution is much smaller than the subscale solution computed using the quasi-static discontinuous subscale model at intermediate times, yet the obtained numerical results for the coarse-scale solution field are qualitatively and quantitatively similar.  The coarse-scale solution field for the dynamic model does exhibit fewer oscillations than the coarse-scale field for the quasi-static model, particularly away from the boundary layer (e.g., near the spatial location $(x,y) = (1,1)$ at time $t = 2.01$).

\begin{figure}[t]
  \begin{center}
      \subfigure[Quasi-static model result at $t=0$]{\includegraphics[width=8cm]{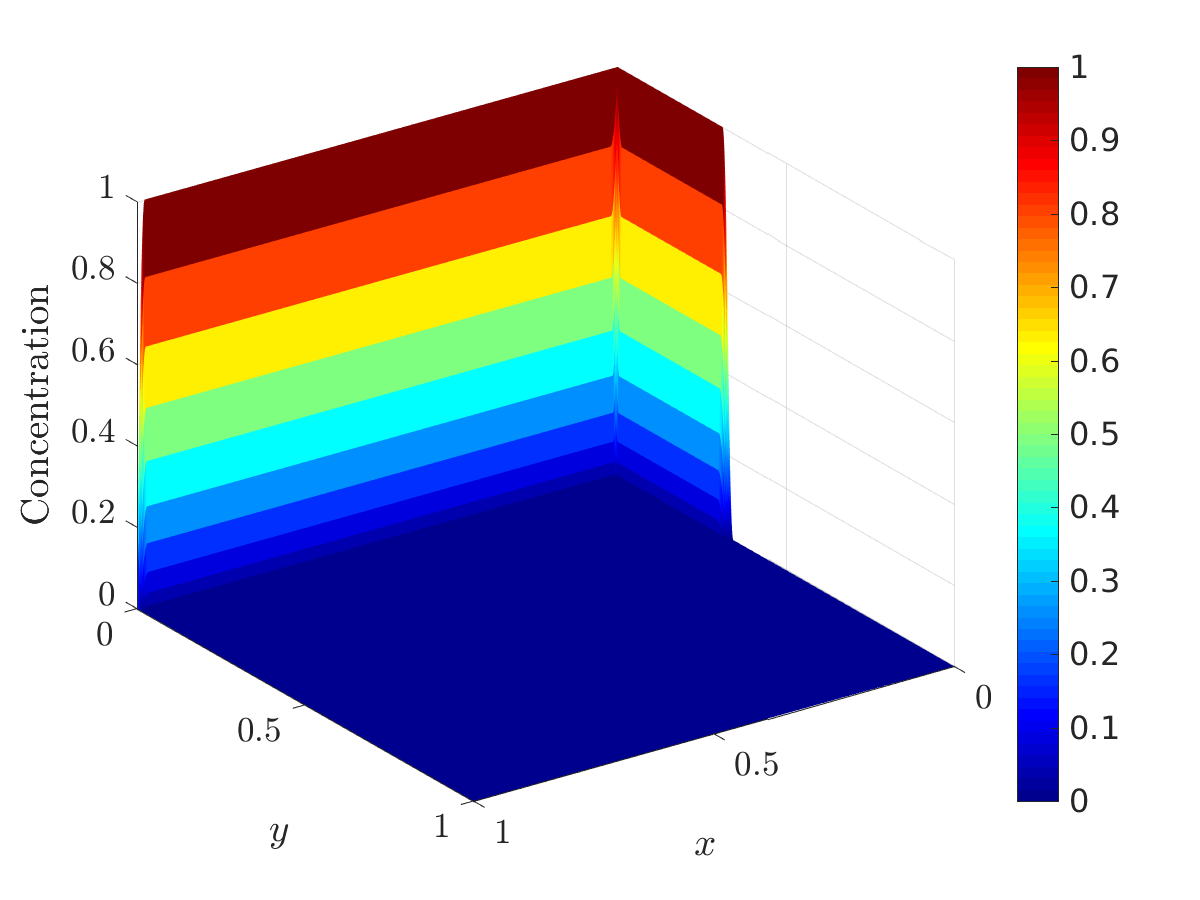}}
      \subfigure[Dynamic model result at $t=0$]{\includegraphics[width=8cm]{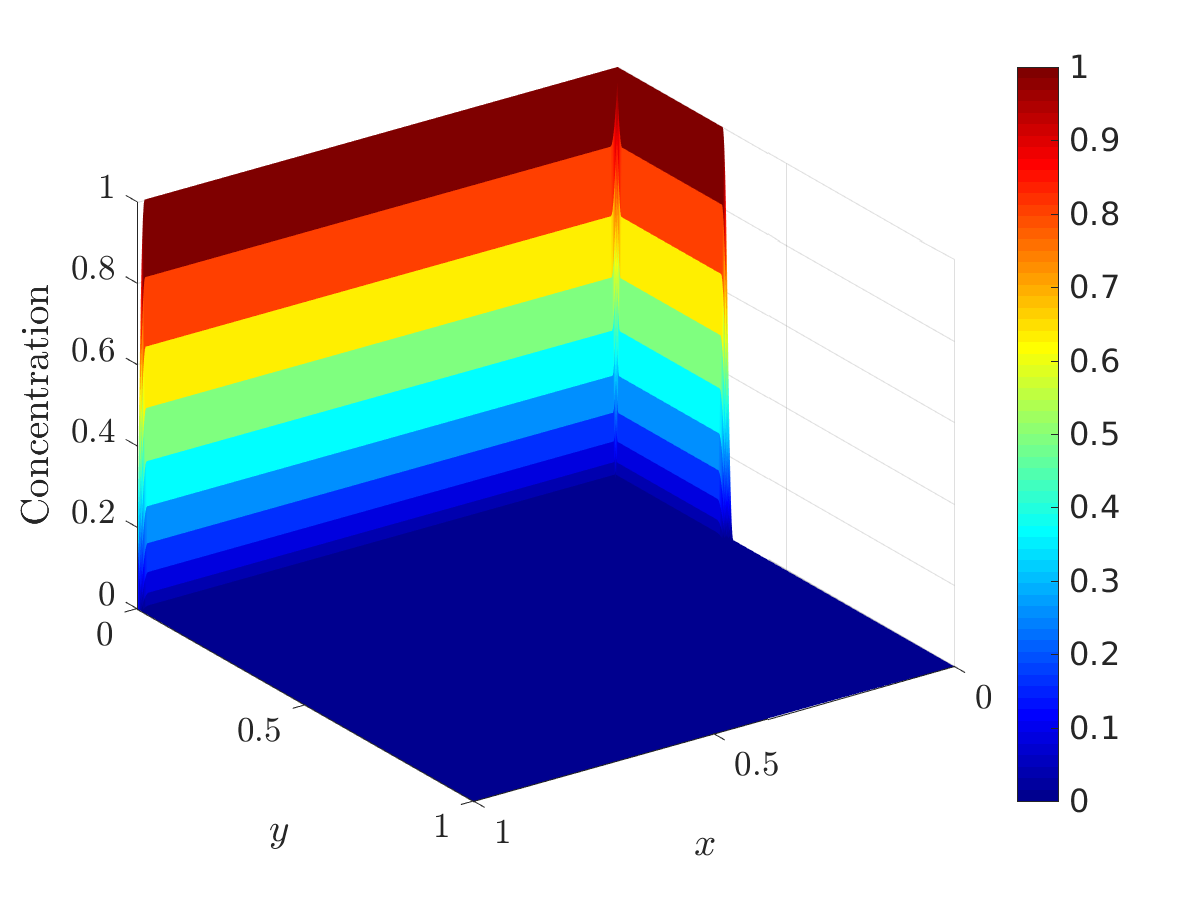}}
      \subfigure[Quasi-static model result at $t=0.51$]{\includegraphics[width=8cm]{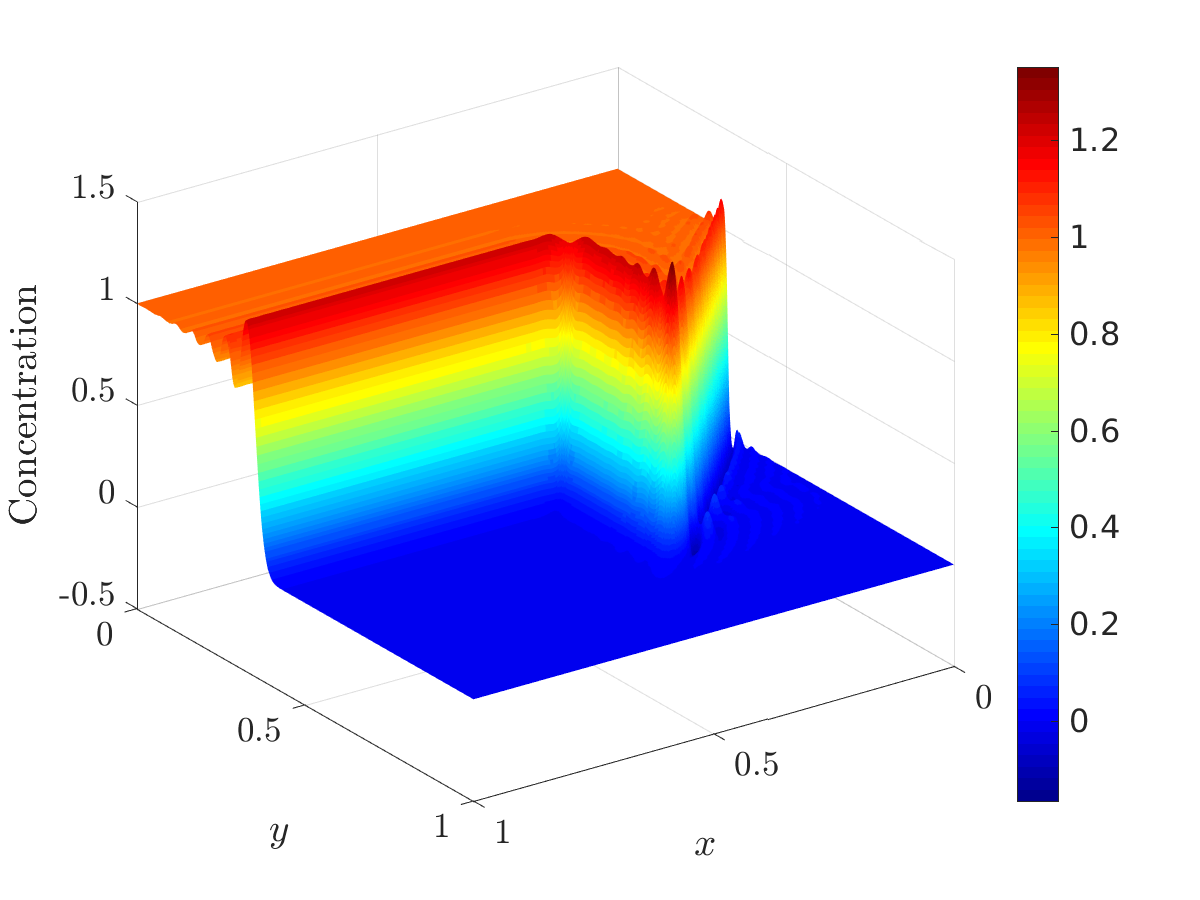}}
      \subfigure[Dynamic model result at $t=0.51$]{\includegraphics[width=8cm]{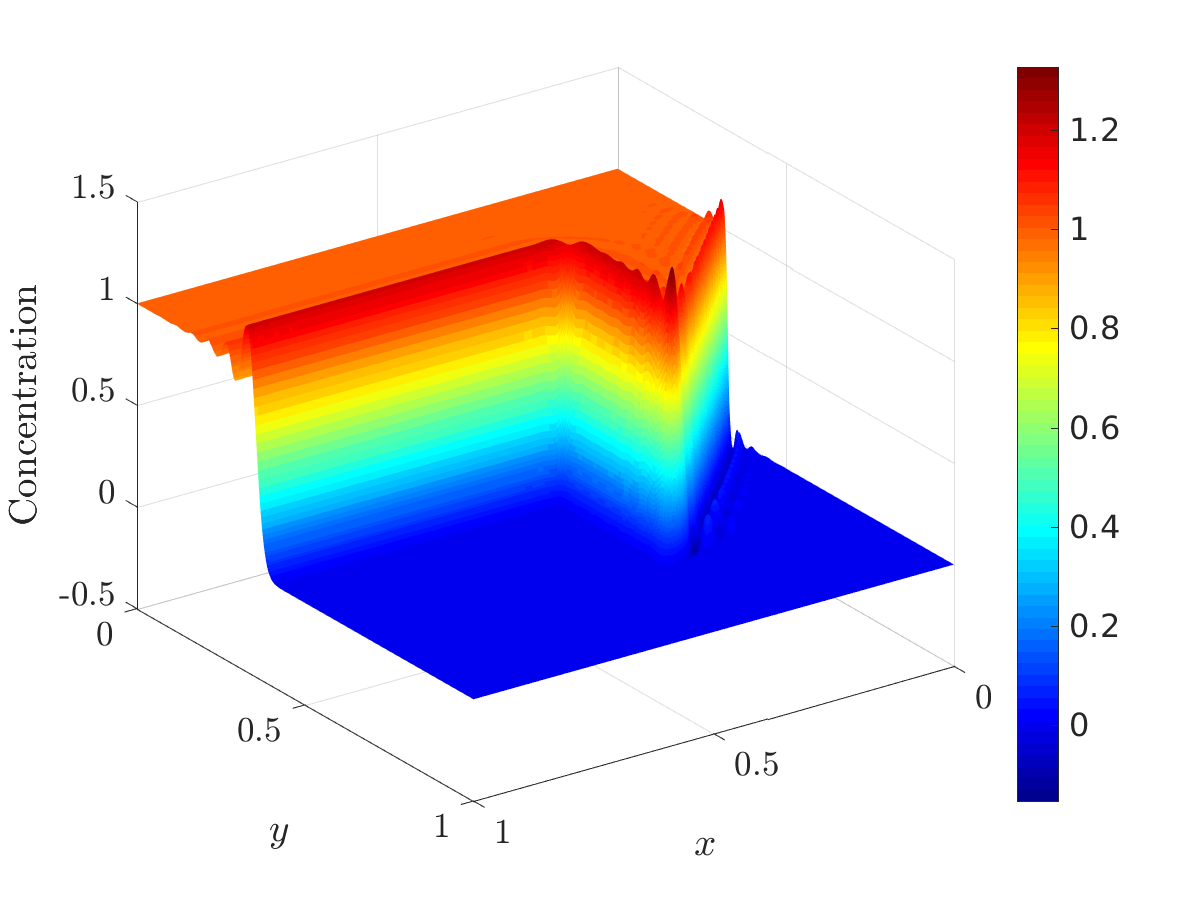}}
      \subfigure[Quasi-static model result at $t=2.01$]{\includegraphics[width=8cm]{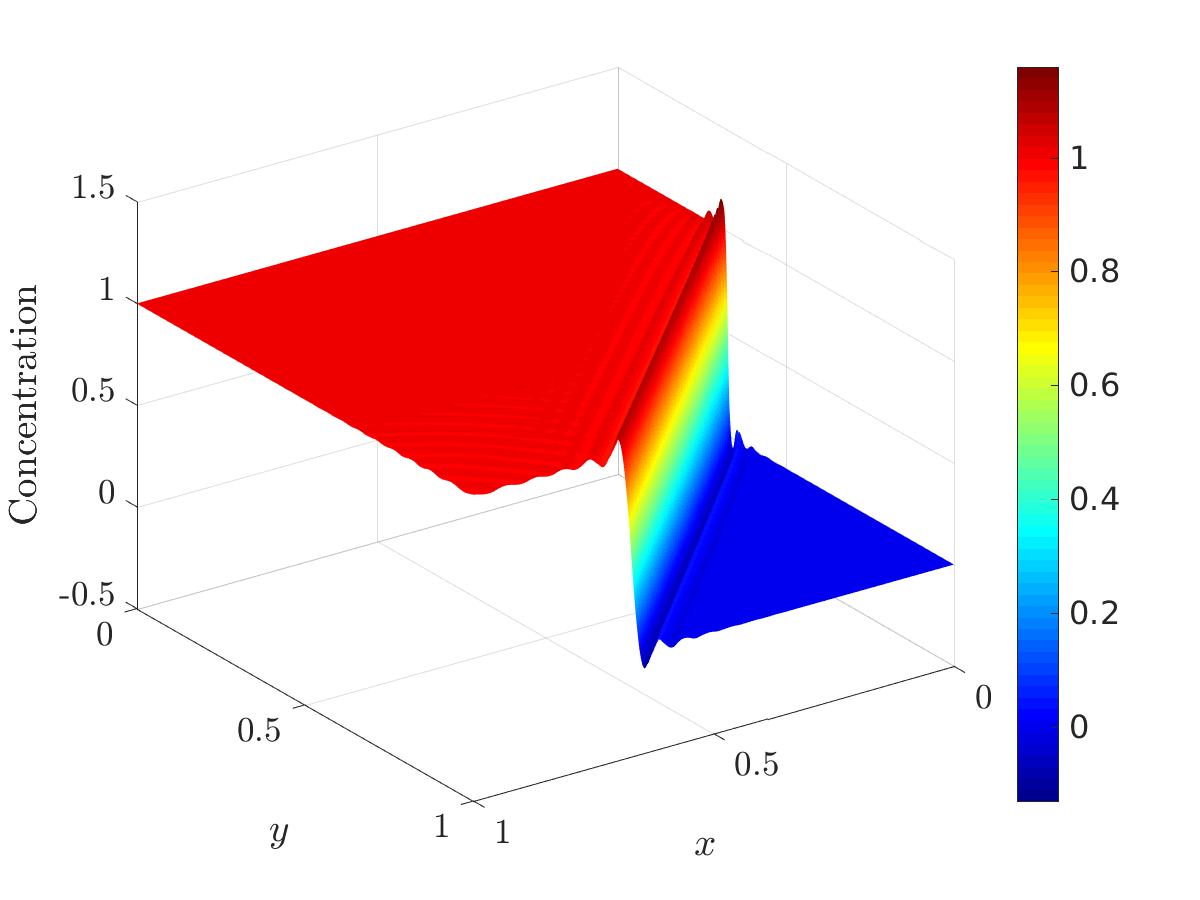}}
      \subfigure[Dynamic model result at $t=2.01$]{\includegraphics[width=8cm]{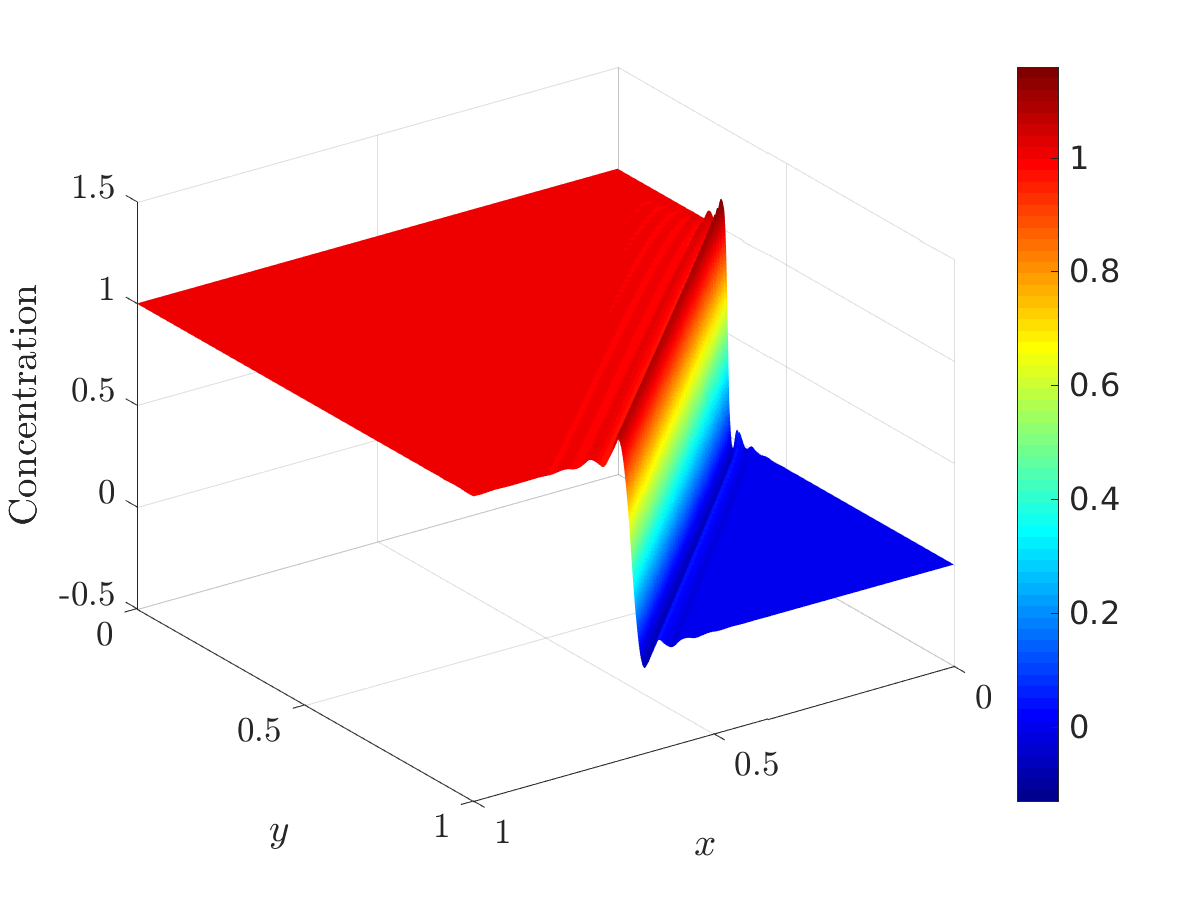}}
	\caption{Coarse-scale solutions for the unsteady advancing front problem for $p = 2$, $p_f = 1$, and $h = 1/64$ at $t = 0, 0.51, 2.01$.}
	\label{fig:advancingfront_coarse}
  \end{center}
\end{figure}

\begin{figure}[t]
  \begin{center}
      \subfigure[Quasi-static model result at $t=0$]{\includegraphics[width=8cm]{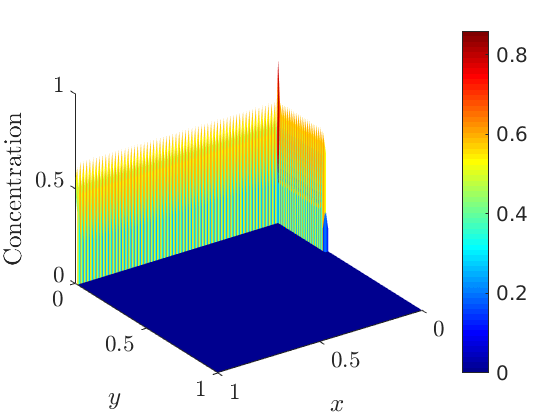}}
      \subfigure[Dynamic model result at $t=0$]{\includegraphics[width=8cm]{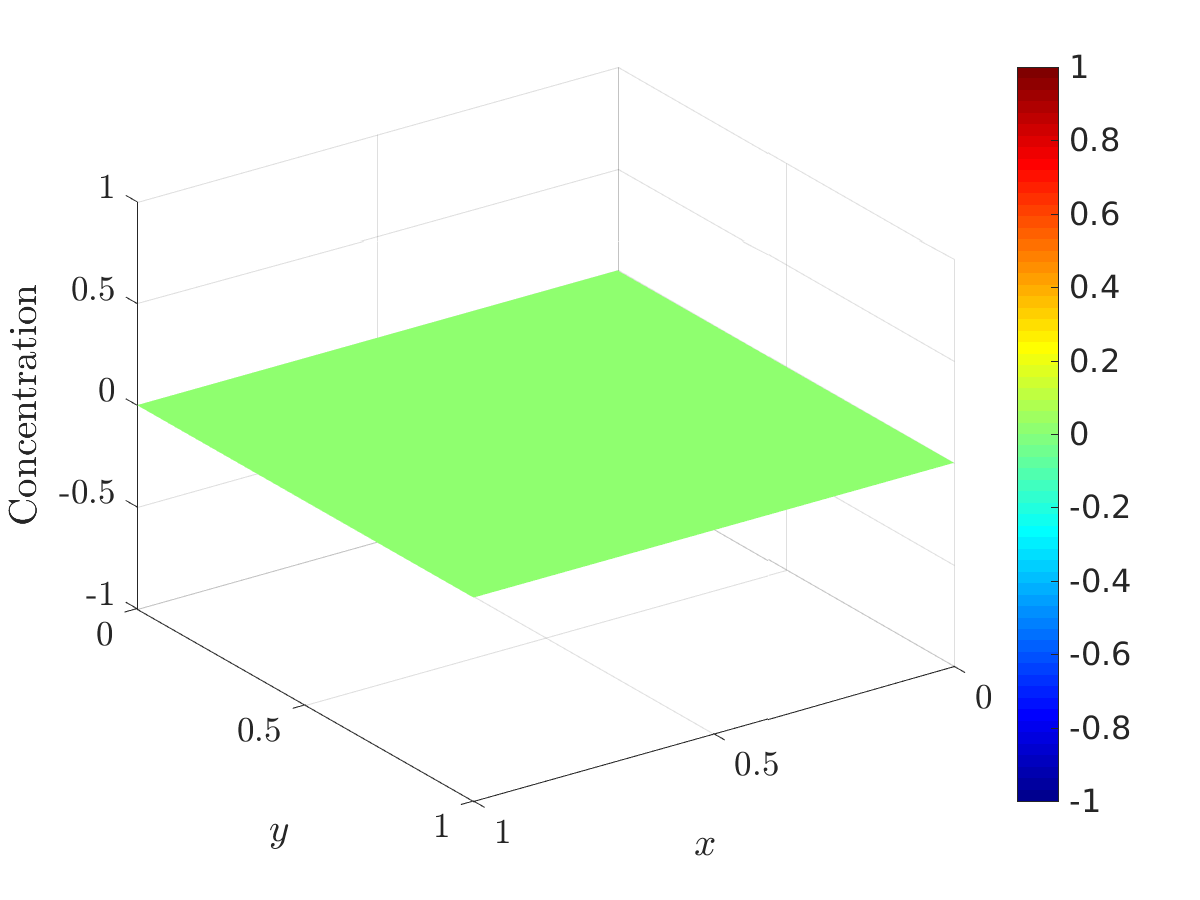}}
      \subfigure[Quasi-static model result at $t=0.51$]{\includegraphics[width=8cm]{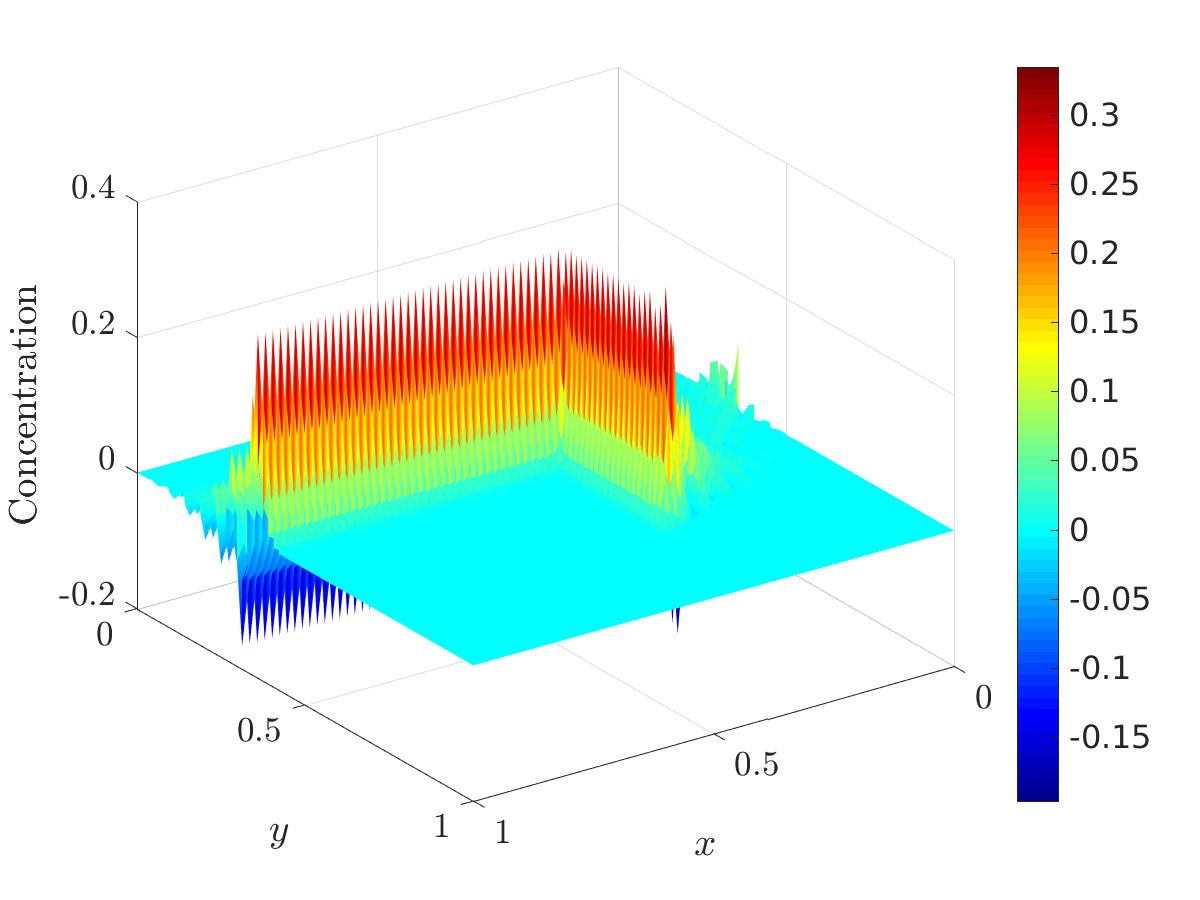}}
      \subfigure[Dynamic model result at $t=0.51$]{\includegraphics[width=8cm]{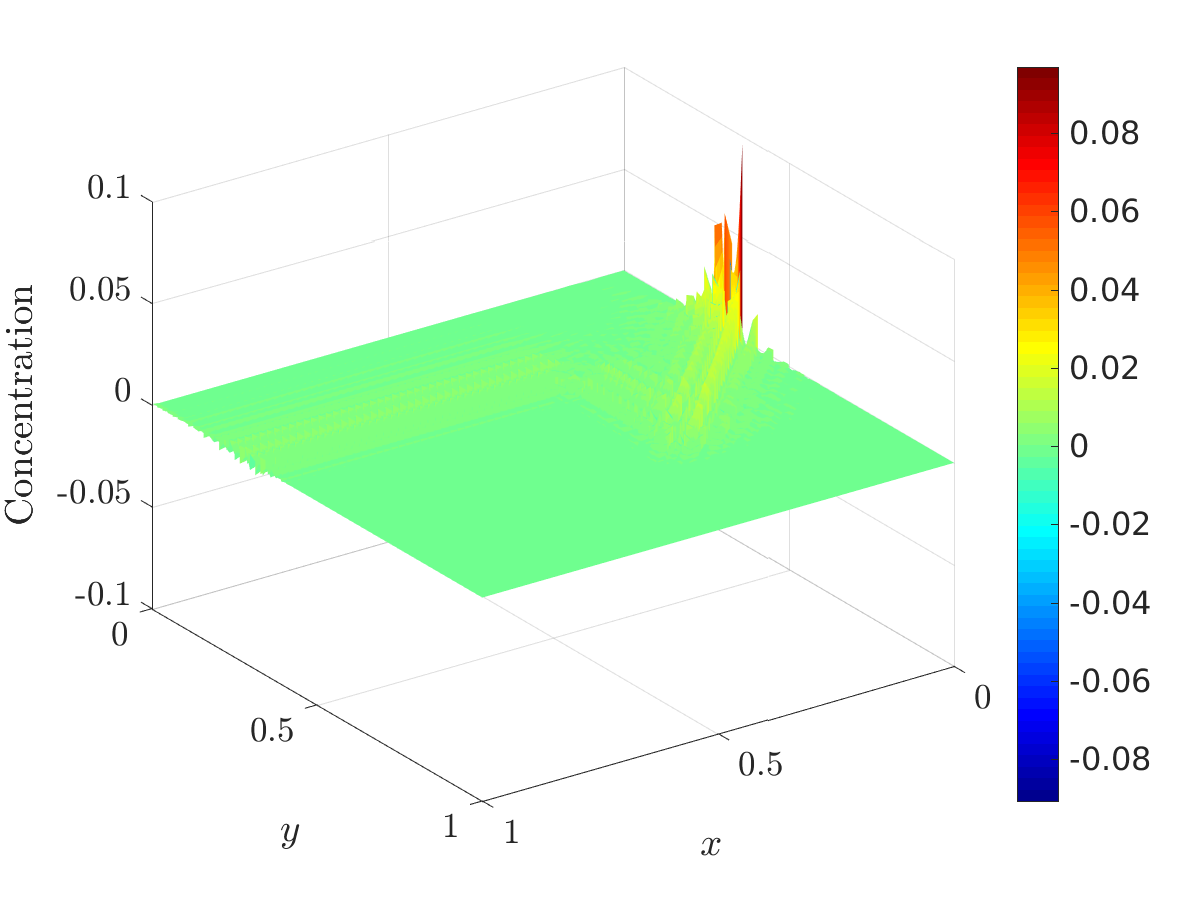}}
      \subfigure[Quasi-static model result at $t=2.01$]{\includegraphics[width=8cm]{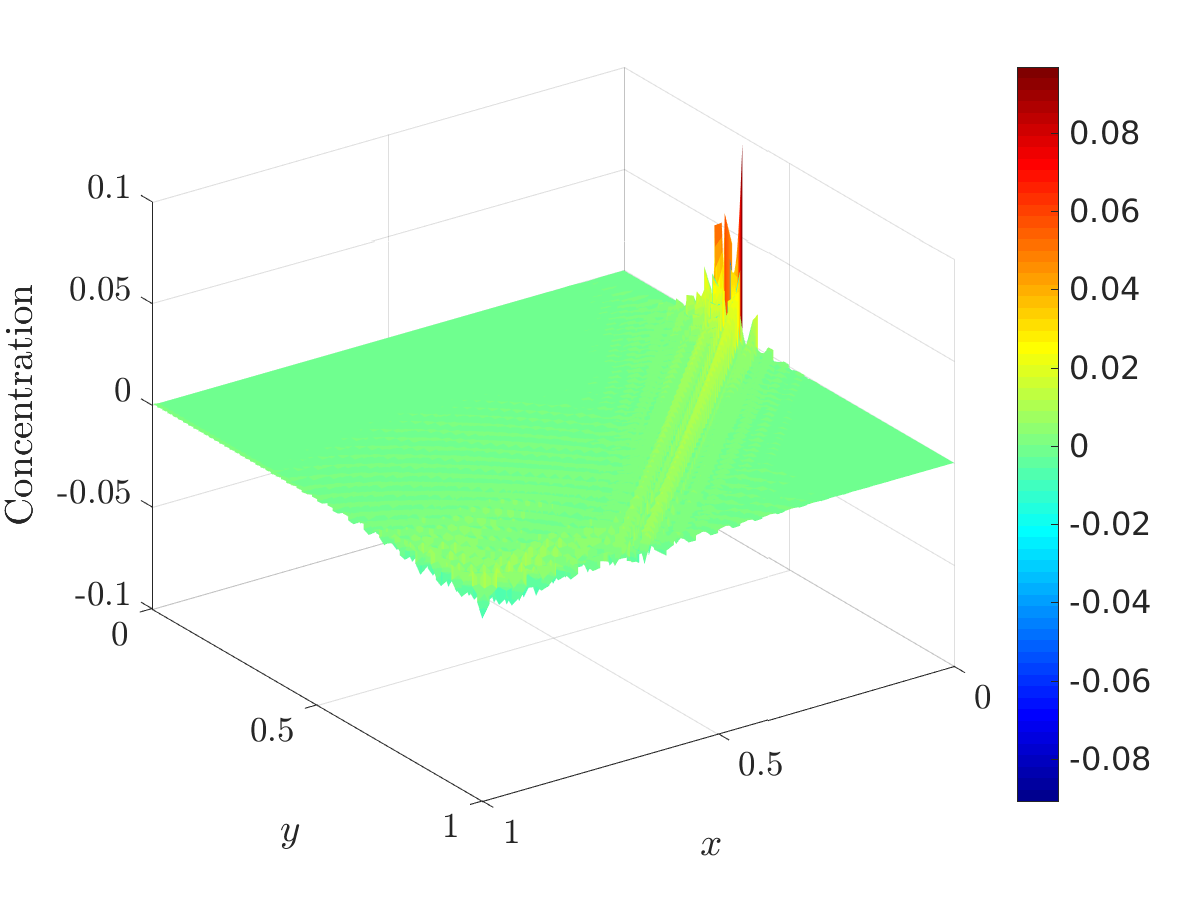}}
      \subfigure[Dynamic model result at $t=2.01$]{\includegraphics[width=8cm]{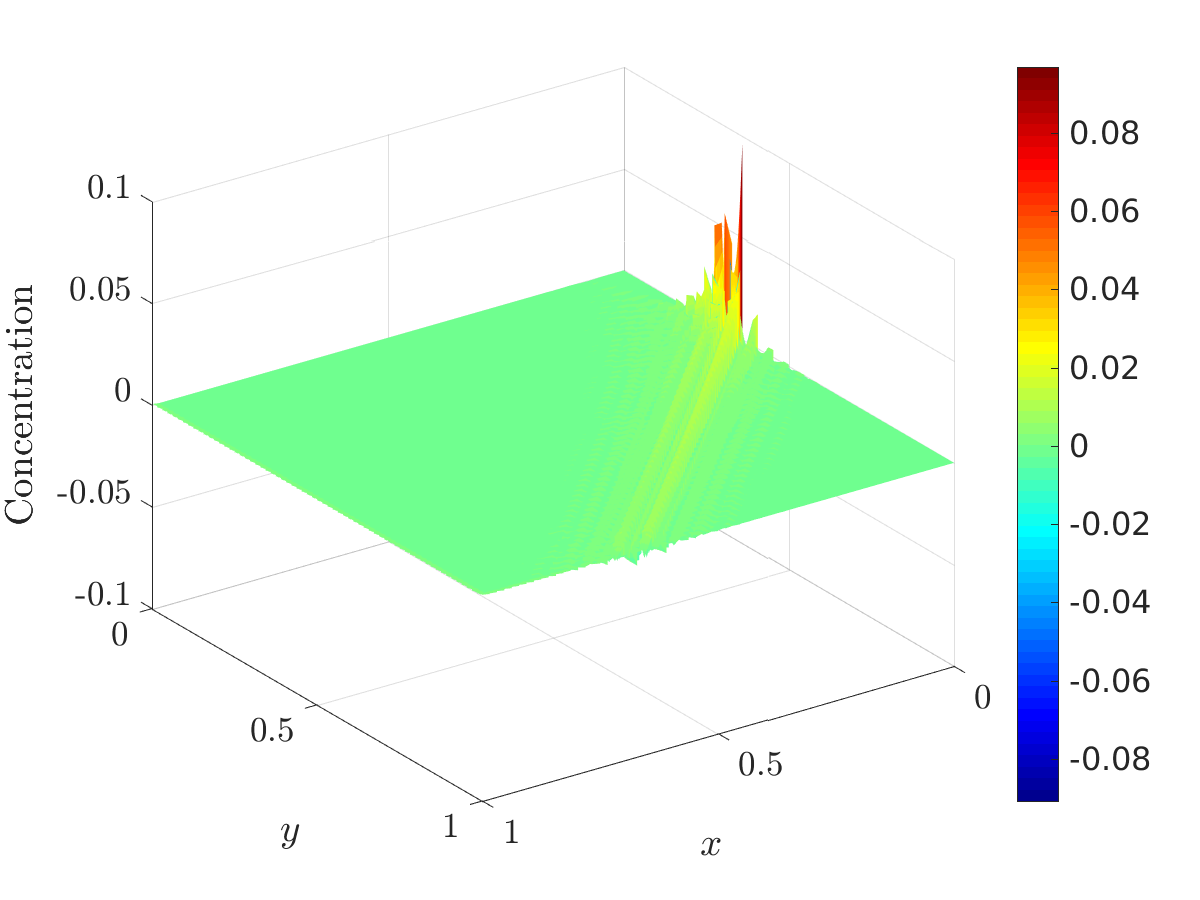}}
	\caption{Subscale solutions for the unsteady advancing front problem for $p = 2$, $p_f = 1$, and $h = 1/64$ at $t = 0, 0.51, 2.01$.}
	\label{fig:advancingfront_fine}
  \end{center}
\end{figure}

\section{Conclusions}

In this paper, we examined a variational multiscale method, which we refer to as the method of discontinuous subscales, that is based on approximating the residual-free bubbles using a discontinuous Galerkin formulation.  We established stability and convergence results for the methodology and demonstrated its applicability to scalar transport problems through a collection of numerical examples.  We demonstrated that the method is accurate, displaying optimal convergence rates with respect to mesh refinement, and stable in the advective-limit.  We also demonstrated that, somewhat surprisingly, lowest-order discontinuous subscale approximations are sufficient to stabilize high-order coarse-scale approximations in the context of isogeometric analysis.

There are two main avenues of research that we plan to pursue in future work.  The first avenue is the design of discontinuous subscale approaches which yield coarse-scale solutions satisfying positivity and monotonicity constraints.  This can be achieved through the use of discontinuity capturing operators \cite{John07,John08} or by directly embedding the constraints within the variational multiscale framework \cite{Evans09}.  The second avenue of research we plan to pursue is the design of discontinuous subscale approaches for LES-type turbulence modeling. While lowest-order discontinuous subscale approximations were found to stabilize coarse-scale approximations in the context of scalar transport, we expect that such approximations will not be sufficient to yield accurate representations of the subgrid stress tensor for turbulent incompressible flows \cite{Wang10}.  We anticipate that improved LES turbulence models may be attained through adaptive refinement of the discontinuous subscale approximation space.

\section*{Acknowledgements}

This material is based upon work supported by the Air Force Office of Scientific Research under Grant No. FA9550-14-1-0113.  The authors would also like to acknowledge early conversations with Thomas J.R. Hughes and J. Austin Cottrell which motivated the subject of this paper.

\bibliographystyle{abbrv-diss}
\bibliography{references}

\end{document}